\documentclass{article}
 
 \usepackage{enumitem}
\usepackage{amsmath}
\usepackage{amsfonts}
\usepackage{amssymb}
\usepackage{amsthm}
\usepackage{graphicx}
\usepackage{multicol}
\usepackage{color}
\usepackage{microtype}
\usepackage{calc}

\title{Connecting orbits for families of Tonelli Hamiltonians}

\author{Vito Mandorino}

\newtheorem{theorem}{Theorem}
\newtheorem{corollary}[theorem]{Corollary}

\newtheorem{lemma}[theorem]{Lemma}
\newtheorem{proposition}[theorem]{Proposition}

\theoremstyle{definition}
\newtheorem{definition}[theorem]{Definition}
\newtheorem{remark}[theorem]{Remark}

\newcommand{\eps}{\varepsilon}

\newcommand{\R}{\mathbb{R}}
\newcommand{\N}{\mathbb{N}}
\newcommand{\T}{\mathbb{T}}
\newcommand{\W}{\mathbb{W}}
\newcommand{\U}{\mathbb{U}}
\newcommand{\V}{\mathbb{V}}

\newcommand{\m}{\mathbb{M}}
\newcommand{\Z}{\mathbb{Z}}
\newcommand{\F}{\mathcal{F}}
\newcommand{\I}{\mathcal{I}}

\newcommand{\G}{\mathcal{G}}
\newcommand{\p}{\mathbb{P}}

\newcommand{\Aa}{\tilde{\mathcal{A}}}
\newcommand{\Mm}{\tilde{\mathcal{M}}}
\newcommand{\Nn}{\tilde{\mathcal{N}}}

\newcommand{\Sic}{{\Sigma^\infty_c}}
\newcommand{\sic}{{\sigma^\infty_c}}

\def\cprime{$'$}

\date{}

\usepackage[hyperindex]{hyperref}

\begin{document}

\maketitle

\begin{abstract}
We investigate the existence of Arnold diffusion-type orbits for systems obtained by iterating in any order the time-one maps of a family of Tonelli Hamiltonians. Such systems are known as `polysystems' or `iterated function systems'.\ When specialized to families of twist maps on the cylinder, our results are similar to those obtained by Moeckel \cite{Moe02} and Le Calvez \cite{LeC07}. Our approach is based on weak KAM theory and is close to the one used by Bernard in \cite{Ber08} to study the case of a single Tonelli Hamiltonian. 
\end{abstract}

\tableofcontents

\section{Introduction}

Much work has been carried out in order to understand the instability properties of Hamiltonian systems, especially for Hamiltonians which are convex in the momenta variables $p$. The basic case of a periodic Hamiltonian defined on the cotangent space $T^*\T\cong\T\times\R$ of the one-dimensional torus $\T=\R/\Z$ corresponds to exact-symplectic twist maps on the cylinder, see \cite{Mos86}. Quite a lot is known in this case, thanks for instance to the original works of Birkhoff \cite{Bir32region,Bir32courbes} and to the KAM and Aubry-Mather theories for twist maps. In particular, a general principle is that the non-contractible invariant circles are the unique obstruction to instability phenomena such as the drift in the $p$-variable.

\medbreak
The situation becomes more complicated when generalizing to higher dimension, namely to Hamiltonians defined on $T^*\T^d$, $d\in\N$, or, more generally, on the cotangent space $T^*M$ of a $d$-dimensional manifold $M$.\ In this setting, among others the variational approach of Mather and Fathi's weak KAM theory has been fruitful, especially in the framework of the so-called Tonelli Hamiltonians. The Mather, Aubry and Mañé sets introduced by Mather and Fathi generalize the invariant circles and the Aubry-Mather sets for twist maps, and provide at the same time both an obstruction and a dynamical skeleton for the instability phenomena. This has allowed a better comprehension of the mechanisms underlying the phenomenon of Arnold diffusion which was firstly exhibited in the seminal paper \cite{Arn64} on a concrete example.

\medbreak
Some studies have also been devoted to the following different generalization: one keeps the dimension $d=1$, and consider instead a family of several twist maps at once, which can be iterated in any order. Following \cite{Mar08}, we shall call such a system a \emph{polysystem},\footnote{The expression \emph{Iterated Functions System} is also used to designate these systems.} and \emph{polyorbits} its (discrete-time) trajectories, see Definition \ref{polyorbit} for more rigour. Of course, the trajectories of a map in the family are also trajectories for the polysystem, thus the polysystem presents at least the same unstable behaviors as the single maps in the family. Nevertheless, one expects new kinds of unstable behavior possibly to be created: some obstructions for a map may be circumvented by non-trivial iterations of other maps in the family. Moeckel \cite{Moe02}, Le Calvez \cite{LeC07} and Jaulent \cite{Jau} have studied this problem, extending some results for single twist maps to the polysystem case. In particular, the general emerging principle is that the unique obstructions to instability phenomena, such as the drift in the $p$-variable, are the common non-contractible invariant circles.

\medbreak
In this paper, we try to merge both generalizations, i.e.\ we deal with a family of several Hamiltonians in arbitrary dimension and we investigate the presence of unstable polyorbits (often we will call them ``diffusion polyorbits'' or ``connecting polyorbits''). More precisely, we will consider the polysystem associated to a family $\F$ of one-periodic Tonelli Hamiltonians\footnote{We recall that a one-periodic Tonelli Hamiltonian is a $C^2$ function $H(x,p,t)$ defined on $T^*M\times\T$ which is strictly convex in $p$ (with positive definite Hessian $\partial^2_p H>0$), superlinear in $p$, and whose Hamiltonian flow is complete.} defined on the cotangent space of a compact $d$-dimensional manifold $M$ without boundary. Just as in the one-dimensional twist map case, one expects that some new unstable behavior may be created by non-trivial iterations of the time-one maps of the family. On the other hand, unlike the single-Hamiltonian case, there is not a definition of Mather, Aubry and Mañé sets for polysystems, hence one may expect the obstructions to come expressed in terms of some more complicated objects.

Our discussion will be in the framework of weak Kam theory, for which we refer to \cite{Fat}. The ideas will be close to those in Bernard's paper \cite{Ber08}, of which the present work may be seen as a generalization to the polysystem case (especially of Section 8 in that paper). We call our method for the construction of unstable polyorbits ``Mather mechanism'', after the paper \cite{Mat93} which introduced some of the basic ideas of the construction. In \cite{Ber08} a slightly different ``Arnold mechanism'' is also presented, more reminiscent of the aforementioned paper \cite{Arn64}.

The results which we obtain are rather abstract in nature: essentially, they give sufficient conditions in order for the diffusion orbits to occur between two cohomology classes (in the sense of Proposition \ref{diffusion intro}). The conditions are encoded, locally around a cohomology class $c$, in a subspace $R(c)$ of ``allowed cohomological directions for diffusion'' (Theorem \ref{theorem intro}). This subspace is in turn defined (cf.\ \eqref{R(c) definition} and Proposition \ref{equivalent expressions}) in terms of some sort of generalized Aubry-Mather sets for the polysystem (the sets $\I_\Phi(\G)$ defined in Remark \ref{remark laxoleinik}(i)), which may be in principle quite difficult to decipher. We believe that the generality of our construction may compensate for this abstract character. Moreover, some further study may lead to more transparent conditions, at least in presence of additional hypotheses. For instance, in the twist map case we are able to recover ``concrete'' and ``optimal'' results (see Corollary \ref{d=1 intro}), similar to those already proved with different methods by Le Calvez and Moeckel, and extending some other results of Mather in \cite{Mat91twist} for a single twist map. 

On the negative side, using a result of Cui \cite{Cui10} we show that, if (in arbitrary dimension) the Hamiltonians in the family commute, our mechanism does not give rise to new instability phenomena, which is somehow expected.

\medbreak
As for the interest in studying Hamiltonian polysystems, let us mention that a motivation lies in the fact that the behavior of some complex single-Hamiltonian systems 
may be to some extent reduced to the analysis of simpler polysystems. We are aware for instance of a work of Bounemoura and Pennamen \cite{BouPen12}, where the polysystem approach is used 
in a neighborhood of an invariant normally hyperbolic manifold, 
and some works of Marco therein cited.

\subsection{Main results}

Before introducing our results, let us review the kind of statements which we want to generalize.

\medbreak

For an exact-symplectic twist map $F$ on the cylinder $\T\times\R$, the archetypal instability result is the following: if, for $A<B$, the annulus $\T\times[A,B]\subset\T\times\R$ does not contain any non-contractible invariant circle, then there exists an orbit $(x_n,p_n)_{n\in\Z}$ such that
$p_{0}<A$ and $p_{N}>B$ for some $N\in\N$. This dates back to Birkhoff \cite{Bir32region, Bir32courbes}, and has been improved in various ways. Two improvements in the framework of Aubry-Mather theory for twist maps will be relevant to us. The first states that if $M_{w_1}$ and $M_{w_2}$ are two Aubry-Mather sets for $F$ of rotation number $w_1$ and $w_2$ respectively, such that there is no non-contractible invariant circle between them, then there exists an orbit $\{z_n=(x_n,p_n)\}_{n\in\Z}\subset \T\times\R$ such that
\[
\alpha\;\text{-}\lim\,z_n\subseteq M_{w_1}\qquad\text{and}\qquad\omega\;\text{-}\lim\, z_n\subseteq M_{w_2}.
\]
The second states that if $(w_i)_{i\in\Z}$ are rotation numbers such that, for any $i$, there is no non-contractible invariant circle between the Aubry-Mather sets $M_{w_i}$ and $M_{w_{i+1}}$, then for every sequence $(\eps_i)_i$ of positive number there exists an orbit which visits in turn the $\eps_i$-neighborhood of $M_{w_i}$. Both these results are due to Mather, we refer to \cite{Mat91twist} for precise statements.

Of course, for a twist map, non-contractible invariant circles do represent obstructions to the drift in the $p$-variable, because they disconnect the cylinder, hence the previous statements are optimal. Therefore the principle stemming from these results is that non-contractible invariant circles are the only obstruction to this kind of instability. 

\medbreak
For a family of exact-symplectic twist maps on the cylinder, the generalization of the Birkhoff result above obtained by replacing in the statement ``non-contractible invariant circle'' with ``common non-contractible invariant circle'' is true. This and other stronger results have been proved by Moeckel, Le Calvez and Jaulent \cite{Moe02,LeC07,Jau}. Again, a common non-contractible invariant circle obviously is a real obstruction to the drift in the $p$-variable, whence the optimality of these results and the principle that, for a polysystem of exact twist maps, the common non-contractible invariant circles are the only obstruction to this kind of instability.

\medbreak
For the case of a single Hamiltonian in higher dimension, usually only sufficient conditions for the existence of unstable orbits can be proved. A great amount of work has been devoted to this topic. Our approach is close to the one of Mather in \cite{Mat93} and of Bernard in \cite{Ber08} (see also \cite{Ber02, CheYan04, CheYan09}). Their results are better expressed in terms of cohomology classes rather than rotation vectors: in their papers, the authors define equivalence relations in $H^1(M,\R)$ such that equivalence between classes implies existence of diffusing orbits between the corresponding Aubry sets. The obstruction for the equivalence is represented, roughly speaking, by the size of the Mañé sets. Notice however that, unlike the one-dimensional case, the obstructions for the equivalence may not always correspond to real obstructions for the dynamics. Nevertheless, if $d=1$ the obstructions to the equivalence turn out to be exactly the non-contractible invariant circles. Therefore, the results on twist maps mentioned above are recovered, and the equivalence relation is then optimal in this case.

\medbreak
The present paper has the same structure: we define (in terms of pseudographs and of the flows of the Hamiltonians in the family $\F$, see Sections \ref{pseudographs} and \ref{forcing definition}) an equivalence relation $\dashv\vdash_\F$ between cohomology classes, which is a natural adaptation to the polysystem case of the relation $\dashv\vdash$ introduced in \cite{Ber08}. We then prove that the occurrence of such a relation implies the existence of diffusing polyorbits, in the sense of Proposition \ref{diffusion intro}. We find sufficient conditions (in terms of the ``homological size'' of some sort of generalized Aubry sets) which ensure, locally around a given class $c$, the occurrence of the relation. If $d=1$, this conditions turn out to be also necessary, hence the relation is optimal in this case. For $\F$ composed by a single Hamiltonian, our results exactly reduce to the one in Section 8 of \cite{Ber08}.

More precisely, let $\F$ be a family of one-periodic Tonelli Hamiltonians on $T^*M$, where $M$ is a $d$-dimensional compact manifold without boundary. For $H\in\F$, we denote by
\[
\phi_H\colon T^*M\to T^*M
\]
the time-one map of the Hamiltonian flow of $H$. Let us first rigorously define what we mean by polyorbit.

\begin{definition}[$\F$-polyorbit]
\label{polyorbit}
A bi-infinite sequence $\{z_n\}_{n\in\Z}\subseteq T^*M$ is a \emph{$\F$-polyorbit} (or, simply, a polyorbit) if for every $n\in\Z$ there exists $H\in\F$ such that 
\[
\phi_{H}(z_n)=z_{n+1}.
\]

A \emph{finite $\F$-polyorbit} (or, simply, a finite polyorbit) is a finite segment $(z_0,z_1,\dots,z_N)$ of a $\F$-polyorbit. We then say that the finite polyorbit joins $z_0$ to $z_N$.

Given two subsets $S, S'\subseteq T^*M$, we say that $S$ is joined to $S'$ by a finite polyorbit if there exist a finite polyorbit joining $z$ to $z'$, for some $z\in S$ and $z'\in S'$.
\end{definition} 
Let $\Aa_H(c)$ and $\Mm_H(c)$ be the Aubry and Mather sets of $H$ of cohomology of $c$, as defined in \cite{Ber08} or \cite{Fat}. Their definition is also recalled in Subsection \ref{wkt}. 

We have (Section \ref{forcing definition}):
\begin{proposition}
\label{diffusion intro}
There exists an equivalence relation $\dashv\vdash_\F$ on $H^1(M,\R)$ such that:
\begin{itemize}
\item[-] if $c\dashv\vdash_\F c'$ then for every $H,H'\in\F$ there exists a polyorbit which is $\alpha$-asymptotic to the Aubry set $\Aa_H(c)$ and $\omega$-asymptotic to $\Aa_{H'}(c')$;
\item[-] if $c\dashv\vdash_\F c'$ and if $\eta,\eta'$ are one-forms of cohomology $c,c'$ respectively, then there exists a finite polyorbit joining $\mathrm{Graph}\,(\eta)$ to $\mathrm{Graph}\,(\eta')$;
\item[-] let $(c_i,H_i,\eps_i)_{i\in\Z}\subset H^1(M,\R)\times\F\times\left]0,+\infty\right[$ such that $c_i\dashv\vdash_\F c_{i+1}$ for every $i$. Then there exists a polyorbit visiting in turn the $\eps_i$-neighborhoods of the Mather sets $\Mm_{H_i}(c_i)$. Moreover, if $(c_i,H_i)=(\bar c,\bar H)$ for $i$ small enough (resp.\ $i$ big enough), then the polyorbit can be taken $\alpha$-asymptotic to $\Aa_{\bar H}(\bar c)$ \textup{(}resp.\ $\omega$-asymptotic to $\Aa_{\bar H}(\bar c)$\textup{)}.
\end{itemize}
\end{proposition}

The main result is Theorem \ref{theorem}. Let us state it here for finite $\F$, even if it will hold under a weaker assumption.

\begin{theorem}
\label{theorem intro}
Assume $\F$ is finite. Then for every $c\in H^1(M,\R)$ there exist a vector subspace $R(c)\subseteq H^1(M,\R)$, a neighborhood $W$ of $c$ and $\eps>0$ such that
\[
c'\ \dashv\vdash_\F\ c' +B_\eps R(c)\qquad\forall\,c'\in W.
\]
\end{theorem}

Of course one needs to have information on the subspace $R(c)$ for the result to be interesting. The definition of $R(c)$ is rather abstract and not too easy to handle (cf.\ the definition given in \eqref{R(c) definition} and some equivalent expressions given in Proposition \ref{equivalent expressions}).

Nevertheless, we are able to prove (Proposition \ref{R(c)=0}) that if there exists a $C^{1,1}$ weak Kam solution of cohomology $c$ which is common to all the Hamiltonians in $\F$, then $R(c)=\{0\}$. In addition, if $d=1$, the viceversa is true: if $R(c)=\{0\}$ then there exists a $C^{1,1}$ weak Kam solution of cohomology $c$ common to all Hamiltonians in $\F$, i.e.\ a common non-contractible invariant circle.

This fact, together with Theorem \ref{theorem intro} and Proposition \ref{diffusion intro} yields the following result for families of twist maps (no additional assumptions on $\F$ will be eventually needed):

\begin{corollary}
\label{d=1 intro}
Let us consider the polysystem associated to an arbitrary family $\F$ of one-periodic Tonelli Hamiltonians on $\T\times\R$. Let us make the identification $H^1(\T,\R)\cong\R$. If, for some $A<B\in\R$, the family $\F$ does not admit an invariant common circle with cohomology in $[A,B]$, then:
\begin{itemize}
\item[(i)] there exists a polyorbit $(x_n,p_n)_{n\in\Z}$ satisfying $p_0 =A$ and $p_N = B$ for some $N\in\N$;
\item[(ii)] for every $H,H'\in\F$ and every $c,c'\in[A,B]$ there exists a polyorbit $\alpha$-asymptotic to the Aubry set $\Aa_{H}(c)$ and $\omega$-asymptotic to $\Aa_{H'}(c')$;
\item[(iii)] for every sequence $(c_i,H_i,\eps_i)_{i\in\Z}\subset [A,B]\times\F\times\left]0,+\infty\right[$ there exists a polyorbit which visits in turn the $\eps_i$-neighborhoods of the Mather sets $\Mm_{H_i}(c_i)$.
\end{itemize}
\end{corollary}

When $d>1$ some information can still be extracted from the subspace $R(c)$. A sample of what can be obtained is presented in Proposition \ref{sample}.  Roughly speaking, among the obstructions which prevent $R(c)$ from being large, we find:
\begin{itemize}
\item[-] for every finite string $H_1,\dots,H_n$ of elements of $\F$, the invariant sets for the map
\[
\phi=\phi_{H_n}\circ\dots\circ\phi_{H_1};
\]
\item[-] for every pair $H_1,H_2$ of elements of $\F$, for every $c$-weak Kam solution $u_1$ for $H_1$ and dual $c$-weak Kam solution $u_2$ for $H_2$, the set
\[
\mathrm{Graph}\, (du_1)\,\cap\,\mathrm{Graph}\, (du_2).
\]
\end{itemize}
However, unlike the twist map case, such obstructions must be intended in a ``negative'' way: their smallness is a sufficient condition for $R(c)$ to be large, the converse being not necessarily true.

\subsection{Structure of the paper}

The paper is organized as follows. In Section \ref{pseudographs} we establish some notation and recall some facts about pseudographs and semiconcave functions. 

In Section \ref{forcing definition} we define the forcing relation $\vdash_\F$ and the mutual forcing relation $\dashv\vdash_\F$, and we show, like in \cite{Ber08}, how the occurrence of such relations implies the diffusion for the polysystem (Proposition \ref{diffusion}).

In Section \ref{section 4} we present the objects needed later to put in place what we call the Mather mechanism: Lagrangian action, Lax-Oleinik operators, operations on costs (minimum, composition) and families of costs. Eventually we build the semigroup $\Sic$ which acts on the space of pseudographs and encodes informations on the underlying polysystem dynamics. The Subsection \ref{wkt} gathers some needed results in weak Kam theory, rephrased in the language of pseudographs.

In Section \ref{mather mechanism} the Mather mechanism for the construction of diffusion polyorbits is put in place. The basic step of the mechanism is proved in Subsection \ref{basic step}. Then we heuristically show the application to the twist map case in Subsection \ref{heuristic twist maps}. Finally, in Subsection \ref{section theorem} we define the subspace $R(c)$ and we prove a general abstract result (Theorem \ref{theorem}) which gives sufficient conditions for the occurrence of the relation $\dashv\vdash_\F$ in terms of $R(c)$. We subsequently apply the result to some special cases (such as twist maps and commuting Hamiltonians), and we discuss the properties of $R(c)$ in relation with the dynamics of the polysystem.


\section{Notation. The space of pseudographs}

\label{pseudographs}

In this section we recall from \cite{Ber08} some facts about pseudographs. We refer to that article for a more detailed introduction. 

Let $M$ be a $d$-dimensional compact connected manifold without boundary. We denote by $\Omega$ the set of smooth closed one-forms on $M$ and by $\pi$ the projection from the cotangent space $T^*M$ to $M$. If $\eta\in\Omega$ we denote by $[\eta]\in H^1(M,\R)$ its cohomology class and, for $S\subseteq\Omega$, $[S]=\{[\eta]:\eta\in S\}$.

If $u\colon M\to\R$ is a Lipschitz function and $\eta\in\Omega$, then the \emph{pseudograph} $\G_{\eta,u}\subset T^*M$ is defined by 
\[
\G_{\eta,u}=\left\{ (x,\eta_x+du_x):x\in M \text{ and } du_x \text{ exists}\right\}.
\]
Let us call $E$ the set of pseudographs:
\[
E=\left\{ \G_{\eta,u}:\eta\in\Omega,u\in {\rm Lip} (M)   \right\}.
\]
Note that 
\[
\G_{\eta,u}=\G_{\eta+df,u-f}
\]
for any smooth function $f\colon M\to\R$. Viceversa, if $\G_{\eta,u}=\G_{\eta',u'}$ then, setting $f=u-u'$, it is not difficult to check that $f$ is smooth,  
$\eta'=\eta+df$ and $u'=u-f$. In particular, the cohomologies of $\eta$ and $\eta'$ are equal. Thus the cohomology of a pseudograph $\G$ is well defined, and we denote it by $c(\G)$. If $\G=\G_{\eta,u}$ for some $\eta$ and $u$, then
\[
c(\G)=[\eta]\in H^1(M,\R).
\]
It is not difficult to see that $E$ is a vector space. In fact, it may be regarded as a quotient of $\Omega\times  {\rm Lip} (M)$ by the subspace $\{(\eta,u):\eta=-du\}=\{(\eta,u):\G_{\eta,u}=\G_{0,0}\}$. The operations of sum and scalar multiplication are explicitly given by
\begin{equation*}
\G_{\eta,u}+\G_{\nu,v}=\G_{\eta+\nu,u+v}\ ,\qquad\qquad\lambda\,\G_{\eta,u}=\G_{\lambda\eta,\lambda u}\qquad\text{for}\ \lambda\in\R
\end{equation*}
(this does not depend on the chosen representatives $(\eta,u)$ and $(\nu,v)$).

We have the following identification of vector spaces, which will be extensively used throughout the paper:
\[
E\ \cong\ H^1(M,\R)\times \bigl({\rm Lip}(M)/\sim\bigr)
\]
where the relation $\sim$ means up to the addition of constants. Given a linear section $S\colon H^1(M,\R)\to \Omega$ (i.e.\ $[S(c)]=c$), an isomorphism performing the above identification is given by
\begin{align*}
H^1(M,\R)\times \bigl({\rm Lip}(M)/\sim\bigr) & \to  E
\\
(c,u)\ \hphantom{{\rm Lip}(M)/\sim} & \mapsto  \G_{S(c),u}.
\end{align*}
The space $E$ can be given a norm via the formula
\[
\|\G_{S(c),u}\|=\| c\|_{H^1}+|u|,
\]
where $|u|$ denotes half the oscillation of $u$, i.e.\ $|u|=(\max u - \min u)/2$. Changing the section $S$ or the norm $\|\cdot\|_{H^1}$ gives rise to an equivalent norm. In the rest of the paper, $S$ and $\|\cdot\|_{H^1}$ will be considered as fixed. Everything will be well-defined regardless of this choice. With a little abuse of language, we will often write $c$ in place of $S(c)$, for instance $\G_{c,u}$ in place of $\G_{S(c),u}$.

We will be mostly concerned with a proper subset of $E$, namely
\[
\p=\left\{ \G_{c,u}:c\in H^1(M,\R), u\colon M\to\R \text{ semiconcave }\right\}.
\]
Here and throughout the paper, the term `semiconcave' stands for the more accurate expression `semiconcave with linear modulus'. Some basic properties of semiconcave functions are reviewed in Subsection \ref{sc}. 

Every $\G\in\p$ is called an \emph{overlapping pseudograph} (the motivation behind this terminology is given in \cite[Section 2.9]{Ber08}). The set $\p$ is closed under sum and multiplication by a positive scalar, but not under difference or multiplication by a negative scalar. In fact, the dual set $\breve\p$ of \emph{anti-overlapping pseudographs} is defined as
\[
\breve\p=-\p=\left\{ \G_{c,u}: c\in H^1(M,\R), u\colon M\to\R \text{ semiconvex } \right\}.
\]
If $c\in H^1(M,\R)$ and $C\subseteq H^1(M,\R)$, the symbols $\p_c$ and $\p_C$ stand for
\[
\p_c=\{\G\in\p:c(\G)=c\},\qquad\p_C=\bigcup_{c\in C}\p_c.
\]
and analogously for $\breve\p_c$ and $\breve\p_C$.
\medbreak

Given a subset $N\subset M$ and a pseudograph $\G$, we denote by $\G_{|N}$ the restriction of $\G$ above $N$, that is $\G\cap \pi^{-1}(N)$.

Given $\G=\G_{c,u}\in\p_c$ and $\breve\G=\G_{c,v}\in\breve\p_c$ (with the same $c$) the set
\[
\G\wedge\breve\G\subseteq M
\]
is defined as the set of the points of minimum of the difference $u-v$. This is a non-empty compact set because $M$ is compact. Moreover, the semiconcavity of both $u$ and $-v$ implies the following property: for every $x$ in $\G\wedge\breve\G$ both $du_x$ and $dv_x$ exist, and they coincide. As a consequence, for any $c$ and any couple $(\G,\breve\G)\in\p_c\times\breve\p_c$, the following definition yields a non-empty subset of $T^*M$:
\[
\G\tilde\wedge\breve\G:=\G_{|\G\wedge\breve\G}=\breve\G_{|\G\wedge\breve\G}= \G\cap\breve\G\cap\pi^{-1}(\G\wedge\breve\G)\subseteq \G\cap\breve\G
\]
and the last inclusion may be strict in general. The set $\G\tilde\wedge\breve\G$ is compact and is a Lipschitz graph over its projection $\G\wedge\breve\G$, by properties of semiconcave functions.

Finally, let us observe that $\Omega$ can be naturally regarded as a subset of $\p\cap\breve\p$. The inclusion is given by $\eta\mapsto\G_{\eta,0}=\mathrm{Graph}\,(\eta)$.

\subsection{Semiconcave functions}

\label{sc}
Let us make a brief digression about semiconcave functions. Recall that for us `semiconcave' means `semiconcave with linear modulus'. We refer to \cite{CanSin04} for a comprehensive exposition in the Euclidean case. On a manifold, the notion of semiconcavity is still meaningful, but the one of semiconcavity constant is chart-dependent. Nevertheless, by taking a finite atlas as shown in \cite[Appendix 1]{Ber08} it is still possible to give meaning to the expression ``$u$ is $C$-semiconcave'' for a real-valued function defined on a compact manifold and $C\in\R$. Hence we can define the best semiconcavity constant of $u$ as
\[
sc(u)=\inf\{ C\in\R: u \text{ is $C$-semiconcave} \}.
\]
It will depend on the particular finite atlas, but this choice will not affect our results. 

We now recall some properties which are well-known in the Euclidean case and which hold true in the manifold case as well. We refer to \cite[Appendix 1]{Ber08} for a more detailed exposition.

We have
\begin{equation}
\label{sc min}
sc\,(\inf_\lambda\{u_\lambda\})\le\sup_\lambda\{sc(u_\lambda)\},
\end{equation}
for any family of functions $\{u_\lambda\}_\lambda$, provided that the infimum is finite. Moreover, if $u_n$ converges uniformly to $u$, then
\begin{equation}
\label{sc liminf}
sc(u)\le\liminf sc(u_n).
\end{equation}
A semiconcave function is differentiable at every point of local minimum (and the differential is $0$). We also have: if $u$ and $v$ are semiconcave and if $x$ is a point of local minimum of $u+v$, then both $u$ and $v$ are differentiable at $x$, and $du_x+dv_x=0$.

A family of functions $\{u_\lambda\}_\lambda$ is called equi-semiconcave if $sc(u_\lambda)\le C$ for some constant $C$ independent of $\lambda$. We will use a lot the following fact: a family of equi-semiconcave functions is equi-Lipschitz (this follows for instance by adapting Theorem 2.1.7 and Remark 2.1.8 in \cite{CanSin04} to the case of a compact manifold).

Finally, the set of semiconcave functions is closed under sum and multiplication by a positive scalar. A function $u$ such that $-u$ is semiconcave is called semiconvex. A function is both semiconcave and semiconvex if and only if it is $C^{1,1}$.


\section{The forcing relation and diffusion polyorbits}
\label{forcing definition}

Let $\F$ be an arbitrary family of one-periodic Tonelli Hamiltonians on $M$. Let us recall that a one-periodic Tonelli Hamiltonian on $M$ is a $C^2$ function
\begin{align*}
H\colon T^*M\times\T&\to\R
\\
(x,p,t)&\mapsto H(x,p,t)
\end{align*}
which is strictly convex and superlinear in $p$ (for any fixed $x$ and $t$) and whose Hamiltonian flow is complete. Our goal is to prove existence of diffusion polyorbits for the family $\F$, in the sense discussed in the Introduction. In this section, we first adapt to the polysystem framework the notion of forcing relation, which was introduced in \cite{Ber08} for the case of a single Hamiltonian. Then we show (Proposition \ref{diffusion}) how this relation implies the diffusion: roughly speaking, if the cohomology class $c$ forces the class $c'$, then there will exist diffusion polyorbits from the cohomology $c$ to the cohomology $c'$, in a sense which will be made precise in the proposition. The aim of the later sections will then be to give sufficient conditions for the forcing relation to occur between two cohomology classes.
\medbreak

Let us recall that we denote by $\phi_H\colon T^*M\to T^*M$ the time-one map of a Tonelli Hamiltonian $H$. We define $\phi_\F$ of a subset $S\subseteq T^*M$ as follows:
\[
\phi_\F(S)=\bigcup_{H\in\F} \phi_H(S),
\]
and we recursively define $\phi_\F^{n+1}(S)=\phi_\F(\phi^n_\F(S))$. Given two subsets $S$ and $S'$ of $T^*M$, we write
\[
S\vdash_{N,\F}S'\quad\overset{\text{def}}{\Longleftrightarrow}\quad S'\subseteq\bigcup_{n=0}^N \phi^n_\F(S).
\]
We write $S\vdash_\F S'$, and we say that $S$ forces $S'$, if $S\vdash_{N,\F}S'$ for some $N\in\N$. Note that, for $z,z'\in T^*M$,
\begin{equation*}
\{z\}\vdash_{\F} \{z'\}\ \Leftrightarrow\ \text{there exists a finite $\F$-polyorbit joining $z$ to $z'$}.
\end{equation*}
(We shall often write $z  \vdash_{\F} z'$ to lighten notations.) In fact, we will mainly interested to the case in which $S=\G$ and $S'=\G'$ are two pseudographs in $\p$. Let us make explicit that
\begin{align*}
\G\vdash_{\F} \G'\ \Leftrightarrow\ &\text{for every $z'\in\G'$ there exists $z\in\G$ and a finite $\F$-polyorbit} 
\\ 
&\text{joining $z$ to $z'$ (with an uniform bound on the length of the polyorbit)}.
\end{align*}
We are now going to extend the definition of $\vdash_\F$ to cohomology classes. If $\W$ and $\W'$ are two subsets of $\p$, we write
\[
\W\ \vdash_{N,\F}\ \W'\quad\overset{\text{def}}{\Longleftrightarrow}\quad \forall\,\G\in\W\quad \exists\,\G'\in\W':\G\vdash_{N,\F}\G'.
\]
We write $\W\vdash_\F \W'$, and we say that $\W$ forces $\W'$, if $\W\vdash_{N,\F}\W'$ for some $N\in\N$. If $\W=\p_c$ or $\W=\p_C$, for some $c\in H^1(M,\R)$ or $C\subseteq H^1(M,\R)$, we simply write $c$ or $C$ in place of $\p_c$ or $\p_C$. Similarly for $\W'=\p_c$. So, for instance, if $c$ and $c'$ are two cohomology classes, the relation
\[
c\vdash_{N,\F} c'
\]
means that for every $\G\in\p_c$ there exists $\G'\in\p_{c'}$ such that $\G\vdash_{N,\F} \G'$.

The relation $\vdash_\F$ is reflexive and transitive (between subsets of $T^*M$ as well as between subsets of $\p$, and in particular between cohomology classes as well). In the sequel, it will be useful to consider the symmetrized relation $\dashv\vdash_\F$ on the cohomology classes defined by
\[
c\dashv\vdash_\F c' \qquad\overset{\textrm{def}}{\Longleftrightarrow}\qquad c\vdash_\F c'\quad\text{and}\quad c'\vdash_\F c.
\]
If $c\dashv\vdash_\F c'$, we say that $c$ and $c'$ \emph{force each other}. The following fact follows directly from the definitions.

\begin{proposition}
The relation $\dashv\vdash_\F$ is an equivalence relation on $H^1(M,\R)$.
\end{proposition}

We also have:

\begin{proposition}
\label{prop:connecting result}
Let $c\vdash_\F c'$.\ Then for any $\G\in\p_c$ and any $\G'\in\breve\p_{c'}$ there exists a finite polyorbit joining $\G$ to $\G'$.
\end{proposition}

\begin{proof}
Let us fix $\G\in\p_c$ and $\G'\in\breve\p_{c'}$. Since $c\vdash_\F c'$, there exists $\G''\in\p_{c'}$ such that $\G\vdash_\F \G''$, which means that for every $z''\in\G''$ there exists a finite polyorbit joining $\G$ to $z''$. Note that the intersection $\G'\cap\G''$ is not empty, because it contains the non-empty set $\G'\tilde\wedge\G''$ (see Section \ref{pseudographs}). Taking $z''$ in this intersection, we get a finite polyorbit joining $\G$ to $\G'$.
\end{proof}

We can now restate and prove Proposition \ref{diffusion intro} about the existence of diffusion polyorbits. The proof is essentially the same as in \cite[Proposition 5.3]{Ber08}.

\begin{proposition}
\label{diffusion}
\mbox{}
\begin{itemize}
\item[1.]
Let $c\vdash_\F c'$. Let $H,H'\in \F$ and $\eta,\eta'$ be two smooth closed one-forms of cohomology $c$ and $c'$ respectively. Then:
\begin{itemize}
\item[(i)] there exists a polyorbit which is $\alpha$-asymptotic to $\Aa_{H}(c)$ and $\omega$-asymptotic to $\Aa_{H'}(c')$;
\item[(ii)] there exists a finite polyorbit joining $\mathrm{Graph}\,(\eta)$ to $\mathrm{Graph}\,(\eta')$;
\item[(iii)] there exists a polyorbit $(z_n)_{n\in\Z}$ which satisfies $z_0\in\mathrm{Graph}\,(\eta)$ and is $\omega$-asymptotic to $\Aa_{H'}(c')$;
\item[(iv)] there exists a polyorbit $(z_n)_{n\in\Z}$ which is $\alpha$-asymptotic to $\Aa_H(c)$ and satisfies $z_0\in\mathrm{Graph}\,(\eta')$.
\end{itemize}
\item[2.] Let
\[
(c_i,H_i,\eps_i)_{i\in\Z}\ \subseteq\  H^1(M,\R)\times\F\times\left]0,+\infty\right[
\]
such that $c_i\vdash_\F c_{i+1}$. Then there exists a polyorbit $(z_n)_{n\in\Z}$ which visits in turn the $\eps_i$-neighborhoods of the Mather sets $\Mm_{H_i}(c_i)$. Moreover, if $(c_i,H_i)=(\bar c,\bar H)$ for $-i$ large enough (resp.\ $i$ big enough), then the polyorbit can be taken $\alpha$-asymptotic to $\Aa_{\bar H}(\bar c)$ \textup{(}resp.\ $\omega$-asymptotic to $\Aa_{\bar H}(\bar c)$\textup{)}.
\end{itemize}
\end{proposition}

\begin{proof}[Proof of 1]
The proof of any one of the four statements relies on a suitable application of Proposition \ref{prop:connecting result}. 

Let us start with $(i)$. As it is well-known, there exist a $c$-weak Kam solution $u$ for $H$ and a dual $c'$-weak Kam solution $u'$ for $H'$. It is also known that $u$ is semiconcave and $u'$ is semiconvex. Let us consider the associated pseudographs $\G=\G_{c,u}\in\p_c$ and $\G'=\G_{c',u'}\in\breve\p_c$.\ By Proposition \ref{prop:connecting result}, there exists a finite polyorbit $(z_i)_{i=0}^N$ joining $\G$ to $\G'$. Moreover, by a general property of weak Kam solutions (see \cite[Proposition 4.3]{Ber08}), every point in $\G$ is $\alpha$-asymptotic for the flow of $H$ to $\Aa_H(c)$ and every point in $\G'$ is $\omega$-asymptotic for the flow of $H'$ to $\Aa_{H'}(c')$. This implies that we can extend the finite polyorbit $(z_i)_{i=0}^N$ to a bi-infinite one satisfying the requirements in $(i)$.

For $(ii)$, let us consider $\G=\G_{\eta,0}$ and $\G'=\G_{\eta',0}$. Since $\eta$ and $\eta'$ are smooth, both $\G$ and $\G'$ belong to $\p_c\cap\breve\p_c$. Hence the Proposition \ref{prop:connecting result} immediately yields a finite polyorbit joining $\G=\mathrm{Graph}\,(\eta)$ to $\G'=\mathrm{Graph}\,(\eta')$.

The proof of statements $(iii)$ and $(iv)$ is similar.
\renewcommand{\qedsymbol}{}
\end{proof}

\begin{proof}[Proof of 2]
It is a natural adaptation of the proof in \cite[Proposition 5.3\,(ii)]{Ber08}.
\end{proof}


\section{Lagrangian action and Lax-Oleinik operators}
\label{section 4}

In this section we introduce the objects needed to put in place, in Section \ref{mather mechanism}, the Mather mechanism for the construction of diffusion polyorbits. For this aim, it is more convenient to adopt the Lagrangian point of view on the dynamics rather than the Hamiltonian one.

Let us quickly recall some basic facts about the Lagrangian point of view: to any one-periodic Tonelli Hamiltonian $H\colon T^*M\times\T\to\R$ one can associate a one-periodic Lagrangian $L\colon TM\times\T\to\R$ via the Fenchel-Legendre transform, i.e.\ 
\[
L(x,v,t)=\sup_{p\in T_x^*M} \left\{   p(v)-H(x,p,t)\right\},\qquad x\in M,\,v\in T_x M,\,t\in\T.
\]
The Lagrangian $L$ turns out to be Tonelli as well, in the sense that it is $C^2$,  it is strictly convex in $v$ (with positive definite Hessian $\partial^2_v L>0$), superlinear in $v$, and the associated Euler-Lagrange flow on $TM\times\T$ is complete. This flow is conjugated to the Hamiltonian flow of $H$ on $T^*M\times\T$. Moreover, the Fenchel-Legendre transform applied to $L$ yields $H$ back. We refer to \cite{Fat} for the proofs of all these facts.

Starting from our family $\F$ of one-periodic Tonelli Hamiltonians, we thus dispose of an associated family of one-periodic Tonelli Lagrangians, which we shall still denote with the same symbol $\F$. For instance, depending on context, both expressions $L\in\F$ and $H\in\F$ will be used.

\subsection{Outline of Section \ref{section 4}}

The general idea is to translate the dynamics of the family $\F$ into a simpler dynamics on pseudographs, by means of the characterization of the former in terms of minimal action and Lax-Oleinik operators.
\smallbreak

\emph{Subsection \ref{action}}. We recall the definition of time-one action of a Tonelli Lagrangian, along with the properties which are important for the sequel (Proposition \ref{action properties}).
\smallbreak

\emph{Subsection \ref{subsect:lax-oleinik}}. We associate to the time-one action (and, more generally, to any cost, i.e.~any continuous function on $M\times M$) a Lax-Oleinik operator, in the usual way. This Lax-Oleinik operator can be interpreted also as an operator on pseudographs (formula \eqref{well-def2}). 

The properties of the time-one action, which have been recalled in Subsection \ref{action}, nicely reflect in properties of the corresponding Lax-Oleinik operator 
(Remark \ref{remark laxoleinik}). These nice properties are not lost under some operations on costs such as minimums and compositions (Proposition \ref{preserved}).
\smallbreak

\emph{Subsection \ref{wkt}}. We review how the language of pseudographs allows to concisely rephrase some aspects of the weak Kam theory for one Tonelli Lagrangian. From the viewpoint of the present article, this may be regarded as a special case in which our family $\F$ is a singleton.
\smallbreak

\emph{Subsection \ref{Sic}}. We generalize the Subsection \ref{wkt} to the general case in which $\F$ is not a singleton. The key object is, for every cohomology $c$, a large semigroup $\Sic$ (depending on $\F$) of Lax-Oleinik operators on the space $\p_c$ of pseudographs of cohomology $c$. This semigroup is essentially the one generated by the time-one actions of the Lagrangians in $\F$ with respect to the operations on costs introduced in Subsection \ref{subsect:lax-oleinik}. As we will see, the dynamics on $\p_c$ of the semigroup $\Sic$ is related to the dynamics on $T^*M$ of the semigroup generated by the time-one maps $\phi_H,\,H\in\F$. 
\smallbreak

Crucially, the semigroup $\Sic$ will contain, after passing to the limit, the operators associated to the Peierls barriers of the Lagrangians in $\F$, along with their successive compositions. This aspect, together with the possibility of ``shadowing'' these operators with ``finite-time'' ones, will be at the heart of the Mather mechanism in the next section.

\subsection{Properties of the Lagrangian action}
\label{action}

Given a one-periodic Tonelli Lagrangian $L$ on $M$ and a closed smooth one-form $\eta$, the time-one action $A_{L,\eta}\colon M\times M\to\R$ is defined by
\begin{equation}
\label{time-one action}
A_{L,\eta}(y,x)=\min_{\substack{\gamma(0)=y, \gamma(1)=x}}\ \int_0^1 L\bigl(\gamma(t),\dot\gamma(t),t\bigr)-\eta_{\gamma(t)}(\dot\gamma(t))\,dt
\end{equation}
where the minimum is taken over absolutely continuous curves $\gamma$. It is well-known that minimizers exist. The following important properties of $A_L$ are also known.

\begin{proposition}
\label{action properties}
\mbox{}
\begin{enumerate}
\item[(i)] $A_{L,\eta+df}(y,x)=A_{L,\eta}(y,x)+f(y)-f(x)$; this is immediate from the definition.
\item[(ii)] $\eta\mapsto A_{L,\eta}$ is continuous if $\Omega$ is endowed with the topology induced from the space of pseudographs $E$ introduced in Section \ref{pseudographs}.

In view of $(i)$ above, this is equivalent to the continuity of $c\mapsto A_{L,S(c)}$. For a proof of this last fact, see \cite[Appendix B.6]{Ber08}. 
\item[(iii)] $A_{L,\eta}$ is semiconcave. Even more, if $C\subset H^1(M,\R)$ is compact, then $\{A_{L,S(c)}\}_{c\in C}$ is equi-semiconcave (for a proof see \cite[Appendix B.7]{Ber08}).

\item[(iv)] $\partial_x A_{L,\eta}(y,x)$ exists if and only if $\partial_y A_{L,\eta}(y,x)$ exists and in this case we have
\[
\bigl(x,\eta_x + \partial_x A_{L,\eta}(y,x)\bigr)=\phi_H\bigl(y,\eta_y-\partial_y A_{L,\eta}(y,x)\bigr),
\]
where $H$ is the Hamiltonian associated to $L$.

\end{enumerate}
\end{proposition}
The time-$n$ action $A^n_L$ is defined by letting $A^1_L=A_L$ and by induction
\[
A^{n+1}_{L,\eta}(y,x)=\min_{z\in M}\, \bigl\{A_{L,\eta}^{n}(y,z)+A_{L,\eta}^1(z,x)\bigr\}
\]
or, equivalently,
\[
A^{n}_{L,\eta}(y,x)=\min_{\substack{\gamma(0)=y,\gamma(n)=x}}\ \int_0^n L\bigl(\gamma(t),\dot\gamma(t),t\bigr)-\eta_{\gamma(t)}(\dot\gamma(t))\,dt,
\]
the minimum being over absolutely continuous curves.

It is well-known that, given $L$, there exists an unique function $\alpha\colon H^1(M,\R)\to\R$ such that the function
\[
h_{L,\eta}(y,x)=\liminf_{n\to\infty}\ A^n_{L,\eta}(y,x)+n\alpha([\eta])
\]
is real-valued for every $\eta$; the family $h_L\equiv\{h_{L,\eta}\}_\eta$ is called the Peierls barrier of $L$. It clearly satisfies the property $(i)$ of Proposition \ref{action properties}; it also satisfies the property $(iii)$, the proof of this fact will be recalled in Subsection \ref{wkt}.


\subsection{Lax-Oleinik operators}
\label{subsect:lax-oleinik}

For any compact space $X$, the set of real continuous functions $C(X)$ will be endowed with the standard sup-norm $\|\cdot\|_\infty$. Any continuous function $A\in C(M\times M)$ will be called a \emph{cost}. We will regard the time-one actions $A_{L,\eta}$ of the previous subsection as very special costs.

To any cost $A$, it is possible to associate the Lax-Oleinik operator $T_A\colon C(M)\to C(M)$ defined by
\[
T_A u(x)=\min_{y\in M}\,\bigl\{ u(y)+A(y,x)\bigr\},\qquad u\in C(M)
\]
and the dual Lax-Oleinik operator $\breve T_A\colon C(M)\to C(M)$
\[
\breve T_A u(y)=\max_{x\in M}\,\bigl\{ u(x)-A(y,x)\bigr\},\qquad u\in C(M).
\]
We call $\I_A(u)\subseteq M$ the set of points $y$ such that $T_Au(x)=u(y)+A(y,x)$ for some $x$.
Let us now list without proof some basic properties of these objects. Note that basically every property of $T_A$ has a dual counterpart in a property of $\breve T_A$, even though we do not always explicit it. Recall that $|\cdot|$ indicates half the oscillation of a function.
\begin{proposition}
\label{laxoleinik properties}
Let $A$ be a cost and $u$ be a continuous function on $M$.
\begin{enumerate}
\item [(i)]  The minimum and the maximum in the above formulas for $T_A u$ and $\breve T_A u$ are actually achieved; $T_A u$ and $\breve T_A u$ actually belong to $C(M)$;
\item[(ii)]  if $A'$ is another cost and $u'$ another continuous function, then
\begin{equation}
\begin{split}
\label{estimate T}
\| T_{A' }u'-T_A u\|_\infty&\le \|A'-A\|_\infty +\|u'-u\|_\infty,
\\
| T_{A' }u'-T_A u|\ &\le\ |A'-A| +|u'-u|
\end{split}
\end{equation}
\item[(iii)] $\I_A(u)$ is compact and non-empty;
\item[(iv)] \label{semicontinuity}the set-valued function $(A,u)\mapsto \I_A(u)$ is upper-semicontinuos;
\item[(v)] for every $A$ and $u$, we have $\breve T_A T_A u\le u$ and
\[
\I_A(u)=\{ y\in M: \breve T_A T_A u (y)=u(y)\}=\mathrm{arg\,min\ } \bigl\{u-\breve T_A T_A u\bigr\}.
\]
\item[(vi)] for every $A$ and $u$, we have
\[
T_A\breve T_A T_A\, u=T_A\, u\qquad\text{and}\qquad \breve T_A T_A \breve T_A\, u=\breve T_A\, u;
\]
\item[(vii)] if $A$ is semiconcave, then $T_A u$ is semiconcave for any $u$, and $sc(u)\le sc(A)$.
\end{enumerate}
\end{proposition}

We are going to consider families of costs indexed by closed smooth one-forms. Let us give some definitions.

\begin{definition}
\label{Ffamily}

Let $A\equiv\{A_\eta\}_{\eta\in\Omega}$ be a family of costs indexed by the closed smooth one-forms. We say that $A$ is:
\begin{itemize}
\item[(i)]  \emph{geometric} if $A_\eta$ is Lipschitz for every $\eta$ and
\begin{equation}
\label{well-def}
A_{\eta+df}(y,x)=A_\eta(y,x)+f(y)-f(x) \qquad\forall\; f\in C^\infty(M);
\end{equation}
\item[(ii)] \emph{continuous} if
\[
\Omega\ni\eta\mapsto\ A_\eta \quad \text{ is continuous}
\]
when $\Omega$ is endowed with the topology induced from $E$, see Section \ref{pseudographs}. Note that if a family $A$ is geometric, the continuity of $c\mapsto A_{S(c)}$ is sufficient in order to have the continuity of $\eta\mapsto A_\eta$; here $S$ is the linear section chosen in Section \ref{pseudographs};
\item[(iii)] \emph{locally equi-semiconcave} if, for any compact $C\subset H^1(M,\R)$, the family $\{A_{S(c)}\}_{c\in C}$ is equi-semiconcave;
\item[(iv)] \emph{of $\F$-flow-type} if there exists $N\in\N$ such that the following holds:
\begin{multline*}
\text{the partial derivatives $\partial_x A_{\eta}(y,x)$ and $\partial_y A_{\eta}(y,x)$ exist}
\\
\Rightarrow\ \bigl(y,\eta_y-\partial_y A_{\eta}(y,x)\bigr)\ \vdash_{N,\F}\ \bigl(x,\eta_x+\partial_x A_{\eta}(y,x)\bigr)
\end{multline*}
for any $\eta$ and for any $(y,x)$. We say that $A$ is of $N,\F$-flow-type if we want to specify the $N$.
\end{itemize}
If all the above conditions are satisfied, we say for short that $A$ is a \emph{$\F$-family}.
\end{definition}

Observe that the Proposition \ref{action properties} says that the time-one actions $\{A_{L,\eta}\}_\eta, L\in\F$, are $\F$-families. In the next subsection we are going to introduce some operations on costs which will preserve the property of being an $\F$-family. This will allow to use the Lagrangian time-one actions as ``basic bricks'' to build many $\F$-families of costs. 

\medbreak
The utility of $\F$-families comes from the following remark.

\begin{remark}
\label{remark laxoleinik}
\mbox{}
\begin{itemize}
\item[(i)]
If $A\equiv\{A_\eta\}_{\eta\in\Omega}$ is a geometric family of costs then 
\[
T_{A_{\eta+df}}(u-f)=T_{A_\eta}(u)-f.
\]
Hence, an induced operator on pseudographs $\Phi_A\colon E\to E$ is well-defined by
\begin{equation}
\label{well-def2}
\Phi_A(\G_{\eta,u})=\G_{\eta,T_{A_\eta}u}
\end{equation}
as well as its dual counterpart
\[
\breve\Phi_A(\G_{\eta,u})=\G_{\eta,\breve T_{A_\eta}u}.
\]
Note that both operators preserve the cohomology of $\G$, i.e.~$c(\Phi_A(\G))=c(\breve\Phi_A(\G))=c(\G)$. 

If $A'$ is another geometric family of costs, and if $\G=\G_{c,u},\G'=\G_{c',u'}\in E$ are two pseudographs, we have the following inequality: 
\begin{equation}
\label{continuity phi}
\begin{split}
\| \Phi_A(\G)-\Phi_{A'}(\G') \|_E & = \|c-c'\|_{H^1}+ |T_{A_{c}}u-T_{A'_{c'}}u'| 
\\
& \le \|\G-\G'\|_E + |A_{c}-A'_{c'}|
\end{split}
\end{equation}
which follows from \eqref{estimate T}.

In the same spirit, $\I_{A_{\eta+df}}(u-f)=\I_{A_\eta}(u)$, thus the set $\I_{A_\eta}(u)$ is also well-defined on pseudographs, and we will denote it by $\I_A(\G)$ or $\I_{\Phi_A}(\G)$. Items $(v)$ and $(vi)$ in Proposition \ref{laxoleinik properties} translate respectively into
\begin{equation}
\label{obstruction}
\I_{A}(\G)=\G\wedge\breve\Phi_A\Phi_A(\G).
\end{equation}
and
\begin{equation}
\label{three times}
\Phi_A\breve\Phi_A \Phi_A=\Phi_A,\qquad\breve \Phi_A\Phi_A\breve\Phi_A=\breve\Phi_A.
\end{equation}

\item[(ii)] If $\{A_\eta\}_\eta$ is a continuous geometric family, then $\Phi_A$ is continuous thanks to the estimate \eqref{continuity phi}. Moreover, $\I_A(\G)$ is upper-semicontinuous viewed as a (set-valued) function from $E$ to $M$. Indeed, the composition
\[
(\eta,u)\mapsto (A_\eta,u)\mapsto \I_{A_\eta}(u)
\]
is upper-semicontinuous (thanks to Proposition \ref{semicontinuity}(iv)), and this remains true when passing to the quotient space of pseudographs.

\item[(iii)]
If $\{A_\eta\}_{\eta\in\Omega}$ is a locally equi-semiconcave geometric family, then $\Phi_A(\p)\subseteq\p$ and $\Phi_A(\p_C)$ is relatively compact for all compact $C\subset H^1(M,\R)$. This is a consequence of Proposition \ref{laxoleinik properties}(vii) and the Ascoli-Arzelà Theorem (recall that equi-semiconcave implies equi-Lipschitz). The analogous result holds true for $\breve\Phi_A$.

\item[(iv)]
If $\{A_\eta\}_\eta$ is a $N,\F$-flow-type, locally equi-semiconcave and geometric family of costs, then
\begin{equation}
\label{basic forcing}
\G_{|\I_{A}(\G)}\ \vdash_{N,\F}\ \Phi_{A}(\G)\qquad\forall\,\G\in\p.
\end{equation}
This important fact is obtained by writing $\G=\G_{\eta,u}$ and then applying Proposition \ref{pointwise basic forcing}. The dual statement is also true and is proved analogously. It can be expressed as
\[
\G_{|\breve\I_{A}(\G)}\ \vdash_{N,-\F}\ \breve\Phi_{A}(\G)\qquad\forall\,\G\in\breve\p.
\]
Here we have denoted by $-\F$ the family $\{-H:H\in\F\}$; its elements are not Tonelli Hamiltonians but the relation $\vdash_{-\F}$ is still meaningful. We have also denoted by $\breve\I_A(\G)$ the set of points $x\in M$ such that $\breve T_{A_\eta} u(y)=u(x)-A_{\eta}(y,x)$ for some $y$ (and $\eta$ and $u$ are such that $\G=\G_{\eta,u}$).
\end{itemize}
\end{remark}

\begin{proposition}
\label{pointwise basic forcing}
Suppose that the family of costs $\{A_\eta\}_\eta$ is of $N,\F$-flow-type, locally equi-semiconcave and geometric. Let $u\colon M\to \R$ be semiconcave and  $v=T_{A_\eta}u$. Then, for every $x$ such that $dv_x$ exists and for every $y$ such that $v(x)=u(y)+A_\eta(y,x)$, we have
\begin{align*}
\text{the derivative $du_y$ exists and satisfies}\quad (y,\eta_y +du_y)\ \vdash_{N,\F}\ (x,\eta_x+dv_x).
\end{align*}
\end{proposition}

\begin{proof}
The proof is essentially the same as in \cite[Proposition 2.7]{Ber08} but we report it for completeness. Let $x$ be such that $dv_x$ exists, and let $y$ be such that $v(x)=u(y)+A_\eta(y,x)$. From the definition of $T_{A_\eta}$, one gets that the function $y'\mapsto u(y')+A_\eta(y',x)$ has a minimum at $y$. Being the sum of two semiconcave functions, both of them have to be differentiable at $y$ and
\[
du_{y}+\partial_y A_\eta(y,x)=0.
\]
Similarly, the function $x'\mapsto v(x')-A_\eta(y,x')$ has a maximum at $x$. Since $dv_{x}$ exists by assumption and $-A_\eta$ is semiconvex, we get that $\partial_x A_\eta(y, x)$ exists and
\[
dv_{x}-\partial_x A_\eta(y, x)=0.
\]
Thanks to the $N,\F$-flow-type property we can conclude:
\begin{align*}
(y,\eta_{y}+du_{y})=\bigl(y,\eta_{y}-\partial_y A_{\eta}(y,x)\bigr)\ \vdash_{N,\F}\ \bigl(x,\eta_{x}+\partial_x A_\eta(y,x)\bigr)=(x,\eta_{x}+dv_{x}).
\end{align*}
\end{proof}

\subsection{Operations on costs and families of costs}
\label{subsect:operations}

There are three quite natural operations on costs. For $A,A'$ two costs and $\lambda\in\R$, they are defined as follows:
\begin{align*}
(A,\lambda)&\mapsto A+\lambda&\emph{(addition of constant)}
\\
(A,A')&\mapsto \min\{A,A'\}&\emph{(minimum)}
\\
(A,A')&\mapsto A'\circ A(y,x)=\min_{z\in M}\, \bigl\{A(y,z)+A'(z,x)\bigr\}&\emph{(composition)}.
\end{align*}
It is easily checked that the three of them are continuous in their arguments and that the Lax-Oleinik operators well-behave in the following sense: for $u\in C(M)$, we have
\begin{align}
\label{laxoleinik compatibility} 
T_{A+\lambda}u&=T_A u+\lambda\nonumber
\\
T_{\min\{A,A'\}}&=\min\{T_A u,T_{A'} u\}
\\
T_{A'\circ A}u&=T_{A'}\circ T_A u.\nonumber
\end{align}

We can define the same operations on families of costs: for all $A\equiv\{A_\eta\}_\eta$, ${A'\equiv\{A'_\eta\}_\eta}$ and all functions $\lambda\colon H^1(M,\R)\to\R$, we define
\[
(A+\lambda)_\eta=A_\eta+\lambda([\eta]),\qquad\min\{A,A'\}_\eta=\min\{A_\eta,A'_\eta\},\qquad (A'\circ A)_\eta=A'_\eta\circ A_\eta.
\]

The following proposition shows that these operations preserve the fact of being a $\F$-family.

\begin{proposition}
\label{preserved}
Let $A, A'$ be two $\F$-families of costs, and $\lambda\colon H^1(M,\R)\to\R$ be a continuous function. Then $A+\lambda,\ \min\{A,A'\}$ and $A'\circ A$ are $\F$-families as well. Moreover, the semiconcavity constants are controlled by
\begin{equation}
\label{eq:sc relations}
\begin{aligned}
sc\,(A+\lambda)_\eta&=sc(A_\eta)
\\
sc(\min\{A,A'\}_\eta)&\le \max\{sc(A_\eta),sc(A'_\eta)\}
\\
sc\,(A'\circ A)_\eta&\le \max\{sc(A_\eta),sc(A'_\eta)\}
\end{aligned}
\end{equation}
for each $\eta\in\Omega$.
\end{proposition}

\begin{proof}
We have to verify that the four conditions of Definition \ref{Ffamily} hold true for the families $A+\lambda, \min\{A,A'\}$ and $A'\circ A$. Conditions $(i)$ and $(ii)$ are easy to check. 

In order to prove $(iii)$ (i.e.\ the local equi-semiconcavity), it suffices to prove the three relations \eqref{eq:sc relations}. The first is obvious and the second follows from \eqref{sc min}. For the third, let $\eta\in\Omega$. For any fixed $z$, each of the functions $(x,y)\mapsto A_\eta(y,z)+A'_\eta(z,x)$ is $\max\{sc(A_\eta),sc(A'_\eta)\}$-semiconcave on $M\times M$. This is a general property for functions on $M\times M$ which have the form $f(y)+g(x)$ with $f$ and $g$ semiconcave. Taking the minimum over $z$ yields $(A'\circ A)_\eta$, without deteriorating the semiconcavity constant due to $\eqref{sc min}$. This proves the third relation in \eqref{eq:sc relations}.

As for the condition $(iv)$, i.e.\ the $\F$-flow-type property, it is obvious for $A+\lambda$. Let us prove it for $\min\{A,A'\}$. Let $\eta\in\Omega$ and $x,y$ be such that $\partial_y \min\{A,A'\}_\eta(y,x)$ and $\partial_x \min\{A,A'\}_\eta(y,x)$ exist. If $A_\eta(y,x)<A'_\eta(y,x)$ then locally $\min\{A,A'\}_\eta=A_\eta$ and the $\F$-flow-type property of $\min\{A,A'\}$ reduces to the $\F$-flow-type property of $A$. Similarly if $A'_\eta(y,x)<A_\eta(y,x)$. In the remaining case in which $A_\eta(y,x)=A'_\eta(y,x)=\min\{A,A'\}_\eta(y,x)$, we have by semiconcavity
\begin{align*}
\partial_x \min\{A,A'\}_\eta(y,x)&=\partial_x A_\eta(y,x)=\partial_x A'_\eta(y,x)
\\
\partial_y \min\{A,A'\}_\eta(y,x)&=\partial_y A_\eta(y,x)=\partial_y A'_\eta(y,x).
\end{align*}
thus the $\F$-flow-type property of $\min\{A,A'\}$ reduces to the $\F$-flow-type property of $A$ or $A'$. From these considerations we conclude that $\min\{A,A'\}$ is a $\F$-flow-type family. From the proof just carried out it is also apparent that, if $A$ is of $N,\F$-flow-type and $A'$ is of $N',\F$-flow-type, then $\min\{A,A'\}$ is of $\max\{N,N'\},\F$-flow-type.

\medbreak

It remains to prove the $\F$-flow-type property for $A'\circ A$. Let $\eta,y,x$ be such that $\partial_y (A'\circ A)_\eta(y,x)$ and $\partial_x (A'\circ A)_\eta(y,x)$ exist. Let $z$ be a point of minimum in the expression
\[
\bigl(A'\circ A\bigr)_\eta(y,x)=\min_{z\in M}\ \{ A_\eta(y,z)+A'_\eta(z,x)\}.
\]
Since $A_\eta(y,\cdot)+A'_\eta(\cdot,x)$ is the sum of two semiconcave functions, both $\partial_x A_\eta(y,z)$ and $\partial_y A'_\eta(z,x)$ exist and they satisfy
\[
\partial_x A_\eta(y,z)+\partial_y A'_\eta(z,x)=0. 
\]
Note also that the function $x'\mapsto (A'\circ A)_\eta(y,x')-A'_\eta(z,x')$ has a maximum at $x$. Since $\partial_x(A'\circ A)_\eta(y,x)$ exists by assumption and $-A'_\eta$ is semiconvex, we get that $\partial_x A'_\eta(z,x)$ exists and
\[
\partial_x(A'\circ A)_\eta(y,x)-\partial_x A'_\eta(z,x)=0.
\]
By similar arguments, one gets that $\partial_y A_\eta(y,z)$ exists and
\[
\partial_y(A'\circ A)_\eta(y,x)-\partial_y A_\eta(y,z)=0.
\]
Using the relations just derived and the $\F$-flow-type property of $A$ and $A'$, we finally get 
\begin{align*}
\bigl(y,\eta_y-\partial_y(A'\circ A)_\eta(y,x)\bigr)&=\big(y,\eta_y-\partial_y A_\eta(y,z)\big)\ \vdash_{\F}\ \big(z,\eta_z+\partial_x A_\eta(y,z)\big)
\\
&=\big(z,\eta_z-\partial_y A'_\eta(z,x)\big)\ \vdash_{\F}\ \big(x,\eta_x+\partial_x A'_\eta(z,x)\big)
\\
&=\big(x,\eta_x+\partial_x(A'\circ A)_\eta(y,x)\big).
\end{align*}
This proves the $\F$-flow-type property for $A'\circ A$. From the proof just carried out it is also clear that, if $A$ is of $N,\F$-flow-type and $A'$ is of $N',\F$-flow-type, then $A'\circ A$ is of $(N+N'),\F$-flow-type.
\end{proof}

Note that \eqref{laxoleinik compatibility} implies
\begin{align}
\label{compatibility}
\qquad\Phi_{A+\lambda}=\Phi_A,&\qquad\Phi_{A'\circ A}=\Phi_{A'}\circ\Phi_A\nonumber
\\
\I_{A+\lambda}(\G)=\I_A(\G),\quad\I_{\min\{A,A'\}}(\G)&\subseteq\I_A(\G)\cup\I_{A'}(\G),\quad\I_{A'\circ A}(\G)\subseteq\I_A(\G).
\end{align}
Instead, we are not able to find an analogous formula for $\Phi_{\min\{A,A'\}}$. Let us notice that even if $\Phi_{A+\lambda}=\Phi_A$, the operation of adding a constant is not completely immaterial: it has a role for operators associated to costs such as $\min\{A+\lambda,A'+\lambda'\}$. If $\lambda'-\lambda$ is sufficiently big, then the corresponding operator will be $\Phi_A$, and if $\lambda-\lambda'$ is sufficiently big, the operator will be $\Phi_{A'}$. Intermediate values of $\lambda'-\lambda$ will correspond to intermediate situations.

\subsection{Weak KAM theory}
\label{wkt}

In this subsection we consider the special case $\F=\{L\}$ and we rephrase in the language of pseudographs some standard results in weak Kam theory. Some of them have already been used in Proposition \ref{diffusion}, and some others will be used in Section \ref{mather mechanism}.

An important role in the theory is played by the so-called weak Kam solutions. There are several equivalent definitions for them. The one which we are going to use is: given a Tonelli Lagrangian $L$ and a cohomology class $c$, a $c$-weak Kam solution for $L$ is a solution $u\in C(M)$ of the equation 
\[
u=T_{A_{L,c}}u+\alpha_L(c),
\]
where $A_{L,c}$ is the $c$-time-one action and $\alpha_L\colon H^1(M,\R)\to\R$ is Mather's $\alpha$-function appeared in Subsection \ref{action}. A dual $c$-weak Kam solution is defined as a solution $u\in C(M)$ of the equation
\[
u=\breve T_{A_{L,c}}u-\alpha_L(c)
\]
We say that $u$ is a weak Kam solution (resp.~dual weak Kam solution) if it is a $c$-weak Kam solution (resp.~dual weak Kam solution) for some $c$. The Weak Kam Theorem (cf.\ \cite[Theorem 4.7.1]{Fat}) states that for any Tonelli Lagrangian $L$ and any cohomology $c$ there exists at least one $c$-weak Kam solution and one $c$-dual weak Kam solution. In fact, $\alpha_L(c)$ is the unique constant such that the above equations admit a solution (assuming, as we do, that $M$ is compact).

It is no surprise, in view of the definition of $\Phi_{A_L}$ in \eqref{well-def2}, that the language of pseudographs allows to concisely reformulate these concepts. From that definition it is indeed immediate that:
\begin{equation}
\label{wks fixed point}
\text{$u$ is a $c$-weak Kam solution for $L$ $\ \Leftrightarrow\ $ $\G_{c,u}$ is a fixed point of $\Phi_{A_L}$.}
\end{equation}
In view of this, we shall call weak Kam solutions as well the fixed points of $\Phi_{A_L}$, and $c$-weak Kam solutions the fixed points in $\p_c$. Analogously for dual weak Kam solutions, with $\breve\Phi_A$ in place of $\Phi_A$. Notice that two $c$-weak Kam solutions $u$ and $u'$ differing by a constant correspond to the same weak Kam solution $\G_{c,u}=\G_{c,u'}$.

Another important object in weak Kam theory is the Peierls barrier $h_L$, introduced in Subsection \ref{action}. Let us point out that
\begin{equation}
\label{peierls extended}
h_{L,c}=\lim_{n\to\infty}\lim_{m\to\infty}\min\{A_{L,c}^n+n\alpha,A_{L,c}^{n+1}+(n+1)\alpha,\dots,A_{L,c}^m+m\alpha\},
\end{equation}
and that, by Proposition \ref{preserved}, the families of costs appearing in the right-hand side are locally equi-semiconcave in the sense of Definition \ref{Ffamily}, with a local (in $c$) common bound for their semiconcavity constants. Hence, they have a local (in $c$) common bound for their Lipschitz constants. By the Ascoli-Arzelà theorem, this implies that the two limits are uniform (for any fixed $c$). Since uniform limits preserve semiconcavity constants, we get that the family of costs $h_{L}$ is locally equi-semiconcave in the sense of Definition \ref{Ffamily}. In fact, by \eqref{eq:sc relations} we have $sc(h_{L,c})\le sc(A_{L,c})$ for each $c$. Remark \ref{remark laxoleinik}(iii) thus applies, i.e.~$\Phi_{h_{L}}(\p_C)$ is relatively compact in $\p$ for every compact $C\subset H^1(M,\R)$.

The next proposition reformulates in our language the well-known identities
\begin{align*}
\min_{z\in M}\, \bigl\{ h_{L,c}(y,z)+ A_{L,c}(z,x)+\alpha_L(c)\bigr\}&=h_{L,c}(y,x),
\\
\min_{z\in M}\, \bigl\{ A_{L,c}(y,z)+ h_{L,c}(z,x)+\alpha_L(c)\bigr\}&=h_{L,c}(y,x),
\\
\min_{z\in M}\,\bigl\{ h_{L,c}(y,z)+ h_{L,c}(z,x)\bigr\}&=h_{L,c}(y,x)\qquad \forall\,y,x\in M.
\end{align*}

\begin{proposition}
\label{peierls identities}
Let $h_L\equiv\{h_{L,c}\}_c$ be the family of costs associated to the Peierls barrier of $L$. The following identities hold true:
\begin{align*}
\Phi_{A_L}\circ\Phi_{h_L}&=\Phi_{h_L}
\\
\Phi_{h_L}\circ\Phi_{A_L}&=\Phi_{h_L}
\\
\Phi_{h_L}\circ\Phi_{h_L}&=\Phi_{h_L}
\end{align*}
\end{proposition}

This proposition has important consequences. Indeed, it implies the following characterizations of weak Kam solutions.

\begin{proposition}
\label{tfae wks}
Let $L$ be a Tonelli Lagrangian, $c\in H^1(M,\R)$ and $u\colon M\to\R$ be a continuous function. The following are equivalent:
\begin{itemize}
\item[(i)] $u$ is a $c$-weak Kam solution for $L$;
\item[(ii)] $\G_{c,u}$ is a fixed point of $\Phi_{A_L}$;
\item[(iii)] $\G_{c,u}$ is a fixed point of $\Phi_{h_L}$;
\item[(iv)] $\G_{c,u}$ belongs to the image of $\Phi_{h_L}$.
\end{itemize}
The dual statement obtained by replacing `$c$-weak Kam solution' with `dual $c$-weak Kam solution' and $\Phi$ with $\breve\Phi$ is also true.
\end{proposition}

\begin{proof}
\mbox{}
\\
$(i)\Leftrightarrow (ii)$ has been already pointed out in \eqref{wks fixed point};
\\
$(iii)\Rightarrow (ii)$: let $\G$ be such that $\Phi_{h_L}(\G)=\G$. We then have, by Proposition \ref{peierls identities},
\[
\Phi_{A_L}(\G)=\Phi_{A_L}\Phi_{h_L}(\G)=\Phi_{h_L}(\G)=\G;
\]
$(iii)\Rightarrow(iv)$ is obvious;
\\
$(iv)\Rightarrow (iii)$: let $\G\in\Phi_{h_L}(E)$; then there exists $\G'\in E$ such that $\Phi_{h_L}(\G')=\G$. By Proposition \ref{peierls identities},
\[
\Phi_{h_L}(\G)=\Phi_{h_L}\Phi_{h_L}(\G')=\Phi_{h_L}(\G')=\G;
\]
$(ii)\Rightarrow(iii)$: for a given $\G\in\p$, the set of costs $A$ such that $\Phi_A(\G)=\G$ is closed under addition of constants, finite minima, compositions and uniform limits. From $\Phi_{A_L}(\G)=\G$ and expression \eqref{peierls extended} we thus get $\Phi_{h_L}(\G)=\G$.
\medbreak

The dual statement is proved analogously.
\end{proof}

Since the image of $\Phi_{A_L}$ is contained in $\p$ and the image of $\breve\Phi_{A_L}$ is contained in $\breve\p$, the previous proposition clearly implies that weak Kam solutions belong to $\p$ and dual weak Kam solutions belong to $\breve\p$. In $d=1$, it is known that the non-contractible invariant circles are exactly the pseudographs which are both weak Kam solutions and dual weak Kam solutions. 

The following proposition will be crucial in the proof of Proposition \ref{R(c)=0}. As usual, $H$ denotes the Tonelli Hamiltonian associated to $L$ via the Fenchel-Legendre transform.

\begin{proposition}
\label{wks&dual}
\mbox{}
\begin{itemize}
\item[(i)] A weak Kam solution $\G\subset T^*M$ is invariant for $\phi^{-1}_H$. A dual weak Kam solution is invariant for $\phi_H$; 
\item[(ii)] if $\G$ is a weak Kam solution belonging to $\breve\p$, then automatically $\G$ is a dual weak Kam solution. Analogously, a dual weak Kam solution belonging to $\p$ is a weak Kam solution;
\item[(iii)] if $\G$ is both a weak Kam solution and a dual weak Kam solution, then $\G$ is a Lipschitz graph over $M$ which is invariant for both $\phi_H$ and $\phi_H^{-1}$. 
\end{itemize}
\end{proposition}

\begin{proof}
\mbox{}
\begin{itemize}
\item[(i)]  Let $\G$ be a weak Kam solution, i.e.~$\Phi_{A_L}(\G)=\G$. From Remark \ref{remark laxoleinik}(iv) and Proposition \ref{action properties}(iv) we know that
\[
\G_{|\I_{A_L}(\G)}\ \vdash_{1,\{L\}}\ \Phi_{A_L}(\G),
\]
hence we get
\[
\G= \Phi_{A_L}(\G)\subseteq \phi_H\bigl(\G_{|\I_{A_L}(\G)}\bigr).
\]
Applying $\phi^{-1}_H$ to both sides we get $\phi^{-1}_H(\G)\subseteq\G$, that is the first claim of the statement. The dual claim is obtained analogously, starting from the dual version of Remark \ref{remark laxoleinik}(iv).
\item[(ii)] Let $\G$ be a weak Kam solution belonging to $\breve\p$. We have
\[
\breve\Phi_{A_L}(\G)\subseteq\phi_H^{-1}(\G)\subseteq\G,
\]
where the first inclusion follows from the dual version of Remark \ref{remark laxoleinik}(iv), while the second inclusion follows from part (i) of this Proposition.

It is not difficult to prove that if a pseudograph is contained in another one, then the two must coincide. Thus the inclusion above implies $\breve\Phi_{A_L}(\G)=\G$, that is $\G$ is a dual weak Kam solution, as desired. The dual statement is analogous.
\item[(iii)] Let $\G$ be both a weak Kam and a dual weak Kam solution. It is immediate from part (i) that $\G$ is invariant both in the past and in the future. Moreover, $\G$ has to belong to $\p\cap\breve\p$, hence it is a Lipschitz graph over $M$ (recall that a function both semiconcave and semiconvex is $C^{1,1}$).\qedhere
\end{itemize}
\end{proof}

We have just seen that a weak Kam solution $\G$ is invariant for $\phi_H^{-1}$. Hence the sequence $\phi_H^{-n}(\G)$ is decreasing in $n$. Moreover, one may prove (see \cite{Ber08}, or Proposition \ref{switched} in which we are going to prove some analogous statements in more general situations) that its intersection is a compact invariant set in both past and future, and is given by
\[
\bigcap_{n\in\N} \phi_H^{-n}(\G) =\G_{|\I_{h_L}(\G)}=\G\tilde\wedge\breve\Phi_{h_L}(\G).
\]
where the second equality is just a direct consequence of \eqref{obstruction}.

We now introduce the $c$-Aubry set of $L$, denoted by $\Aa_L(c)$, which appears in Proposition \ref{diffusion}. One of the possible definitions is the following:
\[
\Aa_L(c)=\bigcap\, \bigl\{ \G_{|\I_{h_L}(\G)}: \G \text{ is a $c$-weak Kam solution}\bigr\}\subseteq T^*M.
\]
For a weak Kam solution $\G$, the subset of $M$ given by $\I_{h_L}(\G)=\G\wedge\breve\Phi_{h_L}(\G)$ is also called the projected Aubry set of $\G$.

If $\G\in\p_c$ and $\G'\in\breve\p_c$, it is always true (see Section \ref{pseudographs}) that $\G\tilde\wedge\G'\subseteq T^*M$ is a compact set which is a Lipschitz graph over its projection $\G\wedge\G'\subseteq M$, hence  the same holds true for each of the sets $\G_{|\I_{h_L}(\G)}=\G\tilde\wedge\breve\Phi_{h_L}(\G)$.

It is then clear that $\Aa_L(c)$ is a compact invariant Lipschitz graph over its projection too, being the intersection of compact invariant Lipschitz graphs. It is less obvious from this description, but true, that $\Aa_L(c)$ is non-empty.

Let us denote by $\V_L$ and $\breve\V_L$ respectively the sets of weak Kam solutions and dual weak solutions for $L$. The function $\Phi_{h_L}$ and $\breve\Phi_{h_L}$ are inverse to each other when restricted to these sets. More precisely,
\[
{\Phi_{h_L}\circ\breve\Phi_{h_L}}_{|\V_L}={\rm id},\qquad {\breve\Phi_{h_L}\circ\Phi_{h_L}}_{|\breve\V_L}={\rm id}.
\]
This is due to the formulas \eqref{three times}. A pair of the type $(\G,\breve\Phi_{h_L}(\G))\in\V\times\breve\V$ is, up to a constant, a conjugate weak Kam pair in the sense of Fathi (see \cite{Fat}). Indeed, we see from Proposition \ref{laxoleinik properties}(v) that if $u$ and $\breve u$ are such that $(\G_{c,u},\G_{c,\breve u})\in \V\times\breve\V$ and $\G_{c,\breve u}=\breve\Phi_{h_L}(\G_{c,u})$, then $u-\breve u$ is constant on the Aubry set of $\G$ (this constant is zero if we choose $\breve u=\breve T_{h_L} T_{h_L}u$).

The following property (which has been used in the proof of Proposition \ref{diffusion}) tells us that weak Kam solutions may be seen as a sort of unstable manifolds of the Aubry set of $L$, and dual weak Kam solutions as stable manifolds. 
\begin{proposition}
\label{unstable}
For every $c$-weak kam solution $\G$ and every $z\in\G$, the $\alpha$-limit of $z$ for $\phi^1_H$ is contained in $\Aa(c)$. Analogously, every point in a dual $c$-weak Kam solution is $\omega$-asymptotic to $\Aa(c)$.
\end{proposition}

\begin{proof}
See \cite[Proposition 4.3]{Ber08}.
\end{proof}

Let us now give one of the possible definitions of the Mather set $\Mm_L(c)$: it is the union of the supports of the invariant measures for $\phi_H^1$ which are contained in $\Aa(c)$. It is a compact invariant set. Finally, the following is one of the possible definitions of the Mañé set $\Nn_L(c)$:
\[
\Nn_L(c)=\bigcup\, \bigl\{ \G_{|\I_{h_L}(\G)}: \G \text{ is a $c$-weak Kam solution}\bigr\}.
\]
This can be proved to be a compact invariant set as well. We have
\[
\Mm_L(c)\subseteq\Aa_L(c)\subseteq\Nn_L(c)\subseteq T^*M.
\]
We refer to \cite{Ber02}, \cite{Fat} or \cite{Mat91action} for a detailed analysis.

\subsection{The semigroup $\Sigma^\infty_c$}
\label{Sic}

In this part we somehow generalize the previous subsection to the case of more than one Tonelli Hamiltonian. Let us recall that our final aim is to get informations about the forcing relation $\vdash_\F$, in order to apply Proposition \ref{diffusion}. We have seen that $\F$-families of costs do give us informations about $\vdash_\F$, and that the time-one actions of the Tonelli Lagrangians in $\F$ are $\F$-families of costs. Moreover, the Proposition \ref{preserved} tells us that being a $\F$-family is a property which is preserved by addition of constants, minimums and compositions. This motivates what follows.
\medbreak

Let $\sigma$ be the unique class of families of costs such that:
\begin{itemize}
\item[(i)] $\sigma$ contains $\{A_L:L\in\F\}$; 
\item[(ii)] $\sigma$ is closed under the operations of addition of constants, minimum and composition (defined in Section \ref{subsect:operations});
\item[(iii)] $\sigma$ is the smallest among those classes satisfying (i) and (ii).
\end{itemize}
By property (ii), $\sigma$ is a semigroup under the operation of composition.
 Recall that the operator $\Phi_A$ associated to a $\F$-family of costs $A$ is the one defined by
\[
\Phi_A(\G_{\eta,u})=\G_{\eta,T_{A_\eta}u}.
\]
By Proposition \ref{preserved} we immediately deduce:

\begin{proposition}
\label{sigma properties}
Every family $A\in\sigma$ is a $\F$-family according to Definition \ref{Ffamily}. Hence, all the conclusions of Remark \ref{remark laxoleinik} apply to $A$, and in particular we have
\[
\G_{|\I_A(\G)}\ \vdash_\F\ \Phi_A(\G)\qquad\forall\,\G\in\p.
\]
\end{proposition}
Let us define
\[
\Sigma=\{\Phi_A:A\in\sigma\}.
\]
By the formula $\Phi_{A'}\circ\Phi_A=\Phi_{A'\circ A}$ it is clear that $\Sigma$ is a semigroup with respect to the composition. For a given $c\in H^1(M,\R)$ we define
\[
\sigma_c=\{A_c:A\in\sigma\}
\]
and call $\sigma^\infty_c$ its closure in $C(M\times M)$:
\[
\sigma^\infty_c={\rm cl}\,(\sigma_c).
\]
Let us stress that the elements of $\sigma^\infty_c$ are just costs and not families of costs. It is clear that $\sigma^\infty_c$ is the smallest class containing $\{A_{L,c}:L\in\F\}$ and closed under addition of constants, minimums, compositions and uniform limits. In particular, it is a semigroup for the composition. Let us point out the important fact that the Peierls barriers $h_{L,c}$ belong to $\sigma^\infty_c$ as well, since the limits involved in their definition are uniform (see the discussion after relation \eqref{peierls extended}).

In order to have good compactness properties, we will often make the assumption that $\F$ is equi-semiconcave, according to the following definition:

\begin{definition}
\label{equi-semiconcave}
We say that a family $\F$ of Tonelli Lagrangians is \emph{equi-semiconcave} if, for every fixed $c\in H^1(M,\R)$, the time-one actions $\{A_{L,c}:L\in\F\}$ form an equi-semiconcave set of functions on $M\times M$, that is
\[
\sup_{L\in\F}\,sc(A_{L,c})<+\infty.
\]
\end{definition}

Of course, any finite family $\F$ of Tonelli Lagrangians is equi-semiconcave. If $\F$ is equi-semiconcave, then by the estimates \eqref{eq:sc relations} we have
\[
\sup_{A\in\sigma^\infty_c}\,sc(A)=\sup_{L\in\F}\,sc(A_{L,c})<+\infty.
\]
hence $\sigma^\infty_c$ is an equi-semiconcave set of functions. In particular, it is an equi-Lipschitz set of functions. By the Arzeli-Ascolà theorem, $\sigma^\infty_c$ is then closed under pointwise limits. Being closed under minimums, it is also closed under countable inf and liminf, unless the resulting function is identically $\pm\infty$. Being a separable space, it is actually closed under arbitrary inf and liminf, unless the resulting function is identically $\pm\infty$. The following property is an immediate consequence of the Arzeli-Ascolà theorem:

\begin{proposition}
\label{prop:compact}
If $\F$ is semiconcave, then $\sigma^\infty_c$ modulo addition of constants is compact.
\end{proposition}

\medbreak
Let us now fix $c\in H^1(M,\R)$. Every element $A\in\sigma^\infty_c$ is a cost and not a family of costs, hence in general it is not associated to an operator from $\p$ to $\p$. Nevertheless, we can still define the operator $\Phi_A\colon\p_c\to\p_c$ by
\[
\Phi_A(\G_{c,u})=\G_{c,T_A u}.
\]
Since the costs in $\sigma^\infty_c$ are semiconcave, the image $\Phi_A(\p_c)$ is really contained in $\p_c$. Finally, we define
\[
\Sigma^\infty_c=\{\Phi_A: A\in\sigma^\infty_c\},
\]
which is a semigroup of operators from $\p_c$ to itself.

The semigroups $\sic$ and $\Sic$ will play a central role in the sequel. The next proposition states some of their useful properties. Note that item $(iii)$ below is a sort of shadowing property. Recall that $|\cdot|$ indicates half the oscillation of a function.
\begin{proposition}
\label{Sigma properties}
Let $\F$ be equi-semiconcave.
\begin{itemize}
\item[(i)] For every $A,A'\in\sigma^\infty_c$ and $\G,\G'\in\p_c$, it holds
\begin{equation}
\label{continuity estimate}
\|\Phi_A(\G)-\Phi'_{A'}(\G')\|_\p\le |A-A'|+\|\G-\G'\|_\p
\end{equation}
and in particular every $\Phi\in\Sic$ is $1$-Lipschitz.
\item[(ii)] The function $\I_A(\G)$ is upper-semicontinuous in both $A\in\sic$ and $\G\in\p_c$.
\item[(iii)] For all $\Phi_A\in\Sic,\G\in\p_c$ and $\U$ neighborhood of $\Phi_A(\G)$ in $\p_c$ there exists $\Phi'\in\Sigma$ such that $\Phi'(\G)\in\U$ (in particular $\G\vdash_\F \U$ by Proposition \ref{sigma properties}).
\end{itemize}
\end{proposition}

\begin{proof}
Item $(i)$ is analogous to the estimate \eqref{continuity phi}. Item $(ii)$ is an easy consequence of Proposition \ref{laxoleinik properties} (iv). For item $(iii)$ note that, by the very definition of $\Sic$, there exists a sequence of costs $A_n\in\sigma_c$ converging to $A$ as $n\to+\infty$. Item $(iii)$ then follows from item $(i)$ and the definition of $\Sigma$. 
\end{proof}

In the next proposition we gather some properties of the minimal subsets of the dynamical system $(\p_c,\Sic)$ which will be needed in the next section. We recall that a minimal subset is a compact subset of $\p_c$ which is stable by the semigroup $\Sic$ and which does not contain any proper subset with the same properties. For compact spaces the existence of minimal subsets is a standard Zorn's Lemma argument (actually, for compact metric spaces the Zorn's Lemma is not needed, see the proof in \cite[Theorem 2.2.1]{HasKat02}). Even if $\p_c$ is not compact, this argument can be easily adapted to our case, as the next proposition shows.

Let us remark that the existence of minimal components is the unique point in our construction where the equi-semiconcavity of $\F$ seems to be crucial. In Corollary \ref{d=1} this assumption will be eventually dropped for the case $d=1$.

\begin{proposition}
\label{minimal}
Assume $\F$ is equi-semiconcave.  Then:
\begin{itemize}
\item[(i)] for any $\G\in\p_c$ its orbit $\{\Phi(\G):\Phi\in\Sigma^\infty_c\}$ is compact;
\item[(ii)] there exists a minimal set. In fact, the orbit of any $\G\in\p_c$ contains a minimal set;
\item[(iii)] $\G\in\p_c$ belongs to a minimal component $\m$ if and only if for every $\Phi\in\Sigma^\infty_c$ there exists $\Phi'\in\Sic$ such that $\Phi'\Phi(\G)=\G$; in this case, $\m$ coincides with the orbit of $\G$. In particular, every minimal component $\m$ is transitive: for every $\G,\G'\in\m$ there exists $\Phi\in\Sic$ such that $\Phi(\G)=\G'$.
\end{itemize}
\end{proposition}

\begin{proof}
\mbox{}
\begin{itemize}
\item[(i)] The orbit of a pseudograph $\G\in\p_c$ is the image of the map
\[
\sic\ni A\mapsto \Phi_A(\G).
\]
This map is continuous by the estimate \ref{continuity estimate}. In addition, $\Phi_{A+\lambda}=\Phi_A$ for every constant $\lambda$. By Proposition \ref{prop:compact} we know that $\sic$ modulo addition of constants is compact, thus the image of the map is compact as well.
\item[(ii)] Given $\G\in\p_c$, its orbit is an invariant and compact set, by item (i). By a general result in topological dynamics, it contains a minimal set.
\item[(iii)] Let $\G\in\m$ with $\m$ minimal, and consider $\Phi\in\Sic$. The orbit of $\Phi(\G)$ contains a minimal component by (ii), and is contained in $\m$ because $\m$ is invariant. By minimality of $\m$, the orbit has to coincide with $\m$. Viceversa, suppose that for every $\Phi\in\Sic$ there exists $\Phi'\in\Sic$ such that $\Phi'\Phi(\G)=\G$. We know that the orbit of $\G$ contains a minimal set $\m$ by (ii). The assumption says that every invariant set contained in the orbit of $\G$ must contain $\G$ as well. We deduce that $\m$ coincides with the orbit of $\G$. \qedhere
\end{itemize}
\end{proof}

In order to have a better understanding of the operators in $\Sic$ and the minimal components of $\p_c$, let us now further investigate about these objects in some special cases.

\begin{itemize}
\item[-] \emph{Case $\F=\{L\}$.} This is the case analyzed in Subsection \ref{wkt}. In addition to what already said there, one can show that in this case $\Sic$ is commutative, and that $\Phi\Phi_{h_{L,c}}=\Phi_{h_{L,c}}\Phi=\Phi_{h_{L,c}}$ for every $\Phi\in\Sigma^\infty_c$. Since the image of $\Phi_{h_{L,c}}$ coincides with its fixed points, it is then easy to verify that $\m$ is a minimal component if and only if $\m=\{\G\}$ for some $c$-weak Kam solution $\G$ for $L$.
\item[-] \emph{Commuting Hamiltonians.} If the Hamiltonians in $\F$ commute with each other, i.e.\ their Poisson bracket satisfies 
\[
\{H,G\}+\partial_t H-\partial_t G=0\qquad\forall\,H,G\in\F,
\]
then it is known (see \cite{Cui10} for the time-periodic case and \cite{CuiLi11,Zav10} for the autonomous case) that the associated Lax-Oleinik semigroups commute and that the Hamiltonians in the family share the same weak Kam solutions and the same Peierls barrier which we denote $\{h_c\}_c$. Thus $\Sic$ is commutative and $\Phi\Phi_{h_c}=\Phi_{h_c}\Phi=\Phi_{h_c}$ for every $\Phi\in\Sic$. It is then easy to verify that the minimal components are exactly the $c$-weak Kam solutions for one (hence all) Hamiltonian in $\F$.

\item[-] \emph{General case.} For every $\Phi_A\in\Sic$ it is possible to define an analogous of the Peierls barrier. Indeed, arguing as for the case $A=A_{L,c}$, one can show (see \cite{Zav12}) that there exists a unique real number $\alpha_A$ such that the liminf
\begin{equation}
\label{h_A}
h_A=\liminf_{n\to+\infty} A^n+n\alpha_A
\end{equation}
is real-valued. Exactly as for the Peierls barrier, we have $\Phi_{h_A}\in\Sigma^\infty_c$, and analogous statements to Propositions \ref{peierls identities} and \ref{tfae wks}\emph{(ii)-(iii)-(iv)} hold. In particular the image of $\Phi_{h_A}$ coincides with its fixed points and with the fixed points of $\Phi_A$.

The interpretation of an arbitrary operator in $\Sic$ in terms of Hamiltonian dynamics is not easy. However, something can be said for particular operators. As a sample, let us pick two Hamiltonians $H_1$ and $H_2$ in $\F$, and call $A_1, A_2$ their time-one actions and $h_1, h_2$ their Peierls barriers. In the next two propositions we prove some properties of $\Phi_{A_2\circ A_1}$ and $\Phi_{h_2\circ h_1}$.
\end{itemize}

\begin{proposition}
\label{switched}
Let $H_1,H_2\in\F$, and call $A_1,A_2$ their time-one actions and $\phi_1,\phi_2\colon T^*M\to T^*M$ their time-one maps. Let us also denote $A=A_2\circ A_1$ and $\phi=\phi_2\circ\phi_1$. Let us consider the operator $\Phi_A\colon\p\to\p$. The following hold true:
\begin{itemize}
\item[(i)] for every $\G\in\p$ and every $n\in\N$ it holds
\[
\phi^{-n}\bigl(\Phi_{A^n}(\G)\bigr)\subseteq \G_{|\I_{A^n}(\G)};
\]
\item[(ii)] the fixed points $\G$ of $\Phi_A$ are invariant in the past with respect to $\phi$; more precisely, they satisfy
\[
\phi^{-n}(\G)\subseteq \G_{|\I_{A^n}(\G)}
\]
\item[(iii)] for every fixed point $\G$ of $\Phi_A$, the set $\G_{|\I_{h_A}(\G)}$ is invariant in the past and in the future with respect to $\phi$;
\item[(iv)] for every fixed point $\G$ of $\Phi_A$, every point in $\G$ is $\alpha$-asymptotic to $\G_{|\I_{h_A}(\G)}$ with respect to $\phi$.
\end{itemize}
\end{proposition}

\begin{proof}
\mbox{}
\begin{itemize}
\item[(i)] This is a more precise version of the relation
\[
\G_{|\I_{A^n}(\G)}\ \vdash_{\F}\ \Phi_{A^n}(\G)
\]
of Remark \ref{remark laxoleinik}\,$(iv)$. It follows by a refinement of the proof of Proposition \ref{pointwise basic forcing}, using property $(iv)$ in Proposition \ref{action properties};
\item[(ii)] it is immediate from item $(i)$ since $\Phi_{A^n}(\G)=\G$ for all $n\in\N$;
\item[(iii)] let $\G$ be a fixed point of $\Phi_A$. It is easy to check that the set $\cap_{n}\phi^{-n}(\G)$, if non-empty, is invariant both in the past and in the future. Hence it suffices to show that this intersection is equal to $\G_{|\I_{h_A}(\G)}$. For this aim, let us first notice that
\begin{equation}
\label{I_h=intersection}
\I_{h_A}(\G)=\bigcap_{n}\I_{A^n}(\G).
\end{equation}
Indeed, from $h_A\circ A^n=h_A$ and relations \ref{compatibility}, it follows that the left-hand side is included in the right-hand side. For the reverse inclusion, write $\G=\G_{c,u}$, consider $\bar y$ belonging to the intersection in the right-hand side and let $x_n\in M$ be such that $u(x_n)=u(\bar y)+A^n(\bar y,x_n)+n\alpha_A$. Then by definition of $h_A$ every accumulation point $x$ of the sequence $x_n$ satisfies
\[
u(x)\ge u(\bar y)+h_A(\bar y,x).
\]
Since $\G$ is a fixed point of $\Phi_A$, we also have $u(x)=\min_y \{u(y)+h_A(y,x)\}$. We deduce that the minimum has to be achieved in $\bar y$, and thus $\bar y\in\I_{h_A}(\G)$. This proves \eqref{I_h=intersection}. In order to conclude the proof of item $(iii)$, it suffices to prove that
\[
\bigcap_{n}\phi^{-n}(\G)=\G_{|\bigcap_{n}\I_{A^n}(\G)}.
\]
The left-hand side is included in the right-hand side by item $(ii)$. The reverse inclusion follows from the fact that, if $z\in\G_{|\I_{A^{n+1}}(\G)}$, then $\phi^n(z)\in\G$. This follows from property $(iv)$ in Proposition \ref{action properties} and a refinement of the proof of Proposition \ref{pointwise basic forcing}.
\item[(iv)] Let $z\in\G$. By equation \ref{I_h=intersection}, it suffices to prove that any $\alpha$-limit of $z$ lies in $\G_{|\I_{A^{N}}(\G)}$ for every $N\in\N$. Since $\G_{|\I_{A^{N}}(\G)}$ is a closed set, it suffices to prove that $\phi^{-n}(z)\in\G_{|\I_{A^{N}}(\G)}$ for $n$ large enough. This is indeed true for $n\ge N$ by item $(ii)$.
\qedhere
\end{itemize}
\end{proof}

\begin{proposition}
\label{switched peierls}
Let $H_1,H_2\in\F$, and call $h_1,h_2$ their Peierls barriers. Let us fix $c\in H^1(M,\R)$ and denote $A_c=h_{2,c}\circ h_{1,c}$. Let us consider the operator $\Phi_{A,c}\in\Sic$ and the subsets $\V_{1,c},\V_{2,c},\V_{A,c}\subset\p_c$ constituted respectively by the fixed points of $\Phi_{h_{1,c}},\Phi_{h_{2,c}}$ and $\Phi_{A,c}$. Then $\V_{A,c}$ is contained in $\V_{2,c}$ and is isometric to a subset of $\V_{1,c}$.
\end{proposition}

\begin{proof}
Obviously $\V_{A,c}$ is contained in the image of $\Phi_{A,c}$, which is contained in the image of $\Phi_{h_{2,c}}$, that is $\V_{2,c}$. Moreover, since $\Phi_{A,c}=\Phi_{h_{2,c}}\circ\Phi_{h_{1,c}}$ and $\Phi_{A,c}$ is the identity when restricted to $\V_{A,c}$, we get that $\Phi_{h_{2,c}}$ is a left inverse for $\Phi_{h_{1,c}}$ on $\V_{A,c}$. Since both of them are $1$-Lipschitz (cf.\ Proposition \ref{Sigma properties} (ii)), $\Phi_{h_{1,c}}$ has to be an isometry between $\V_{A,c}$ and $\Phi_{h_{1,c}}(\V_{A,c})$, which is a subset of $\V_{1,c}$.
\end{proof}

Let us point out that, if $d=1$, the whole of $\Sigma^\infty_c$ would not be needed for the purposes of this article. Indeed, the heuristic discussion in Section \ref{heuristic twist maps} as well as the proof of Proposition \ref{R(c)=0} show that the Peierls barrier operators $\Phi_{h_L},L\in\F$, would suffice to get optimal results. Nevertheless, if $d>1$, considering the whole of $\Sigma^\infty_c$ gives stronger (though more abstract) results.


\section{The Mather mechanism}
\label{mather mechanism}

Throughout the whole section, the family $\F$ is assumed to be equi-semiconcave in the sense of Definition \ref{equi-semiconcave}, unless otherwise stated. For a subset $S\subseteq M$, we call $S^\bot\subseteq\Omega$ the vector subspace of the smooth closed one-forms whose support is disjoint from $S$ and $[S^\bot]$ its projection on $H^1(M,\R)$. It follows from the finite dimensionality of $H^1(M,\R)$ that there always exists an open set $U\supseteq S$ such that $[U^\bot]=[S^\bot]$. Such a $U$ will be called an \emph{adapted neighborhood} of $S$. 
Let us point out that, if $M=\T$, we have $[S^\bot]=\{0\}$ if and only if $S=\T$, and otherwise $[S^\bot]=H^1(\T,\R)\cong\R$. For a vector subspace $V\subseteq H^1(M,\R)$, we denote the $\eps$-radius ball centered at the origin by $B_\eps(V)$.

\subsection{Outline of Section \ref{mather mechanism}}

We describe a mechanism for the construction of diffusion polyorbits. When $\F=\{L\}$ is a singleton, our construction essentially boils down to the one in \cite{Ber08}. Some of the main ideas come from the seminal paper of Mather \cite{Mat93}.
\smallbreak

\emph{Subsection \ref{subsect:basic step}}. We prove the technical results which are at the core of the Mather mechanism. Basically, they show how a pseudograph $\G$ may force nearby cohomologies, with the sets $\I_A(\G), A\in\sic$ acting as obstructions to this phenomenon.
\smallbreak

\emph{Subsection \ref{heuristic twist maps}}. We heuristically show how the results of the previous subsection apply to polysystems of twist maps on the cylinder.
\smallbreak

\emph{Subsection \ref{section theorem}}. We apply the results of Subsection \ref{subsect:basic step} to prove a general theorem in arbitrary dimension. We discuss some consequences and applications, including the rigorous counterpart of the heuristic picture in Subsection \ref{heuristic twist maps}. 

\medbreak

Loosely speaking, the mechanism works in the following way: we will be able to associate to every $c\in H^1(M,\R)$ a subspace $R(c)\subseteq H^1(M,\R)$ of ``allowed cohomological directions'' for the forcing relation $\dashv\vdash_\F$. In view of Proposition \ref{diffusion}, this gives allowed cohomological directions for the diffusion: the larger the subspace $R(c)$ is, the more are the directions for which connecting orbits starting at $c$ exist. The obstruction for this subspace to be large will be, roughly, the homological size of the sets $\I_A(\G)$, for $\G\in\p_c$ and $A\in\sic$.

\subsection{The basic step}
\label{subsect:basic step}

Let us introduce some notations: for $c\in H^1(M,\R)$, $\G\in\p_c,A\in\sigma^\infty_c$, we define
\[
R_A(\G)=[\I_A(\G)^\bot]=[\G\wedge\breve\Phi_{A}\Phi_A(\G)^\bot]\subseteq H^1(M,\R).
\]
Here the second equality follows from \ref{obstruction}. More generally, for $A_1,\dots,A_n\in\sic$, we define
\begin{multline*}
R_{A_1,\dots,A_n}(\G)
=\Bigl[\I_{A_n\circ\dots\circ A_1}(\G)^\bot+\I_{A_n\circ\dots\circ A_2}(\Phi_{A_1}(\G))^\bot 
\\
+\dots+\I_{A_n}\bigl(\Phi_{A_{n-1}}\circ\dots\circ\Phi_{A_1}(\G)\bigr)^\bot\Bigr].
\end{multline*}
We will see in Lemma \ref{concatenate} that the subspace $R_{A_1,\dots,A_n}(\G)$ should be intended as a subspace of ``allowed cohomological directions for the forcing relation, through the composition $\Phi_{A_n}\circ\dots\circ\Phi_{A_1}$, starting from $\G$\,''. By taking the union over all finite strings $(A_1,\dots,A_n)$, one should get a space of ``allowed cohomological directions for the forcing relation starting from $\G$\,''. Afterward, by intersecting over all $\G$ in $\p_c$, one should get a space of ``allowed cohomological directions for the forcing relation starting from $c$\,'', which is basically what we are looking for in order to apply Proposition \ref{diffusion}. This motivates the following definitions:
\begin{gather}
\label{R(c) definition}
R(\G)=\bigcup_{\substack{A_1,\dots,A_n\in\sic \\ n\in\N}}R_{A_1,\dots,A_n}(\G)\nonumber
\\
R(c)=\bigcap_{\G\in\p_c} R(\G)
\end{gather}
At this stage it is not clear whether $R(\G)$ or $R(c)$ are vector subspaces. In Proposition \ref{equivalent expressions} several equivalent expressions for $R(c)$ will be given. They will imply that $R(c)$ is indeed a vector subspace, and $R(\G)$ is a vector subspace for every $\G$ in a minimal component of $\p_c$. We shall write $R_\F(c)$ when we want to emphasize the dependence on the family $\F$.

\medbreak
The following lemma is the basic key step in the accomplishment of the Mather mechanism. Given a family of costs $A\in\sigma$,\footnote{We recall that the definition of $\sigma$ is given in Section \ref{Sic}.} the lemma shows how a pseudograph $\G$ may force nearby cohomologies, with the set $\I_A(\G)$ acting as an obstruction to this phenomenon. Furthermore, the semicontinuity in $\G$ of $\I_A(\G)$ allows to extend the conclusion to a whole neighborhood of $\G$.

\begin{lemma}
\label{basic step}
Let $A$ be a $\F$-family of costs according to the Definition \ref{Ffamily} (in particular, $A\in\sigma$ will work). Let $\Phi_A$ be the associated operator on pseudographs. Then, for every $\G\in\mathbb P$ and for every neighborhood $\mathbb U$ of $\Phi_A(\G)$ in $\mathbb P$ there exist $N\in\N$, a neighborhood $\mathbb W$ of $\G$ and an $\eps>0$ such that:
\begin{gather*}
\forall\ \G'\in\mathbb W\,,\quad\forall\ c\,\in\,c(\G')+B_\eps R_{A}(\G) \qquad\exists\, \G''\ \text{ such that}
\\
\G''\in\mathbb U,\qquad\G'\ \vdash_{N,\F}\ \G'',\qquad c(\G'')=c.
\end{gather*}
\end{lemma}

\begin{proof}
Let us fix $\G$, $\mathbb U$ and an adapted neighborhood $U$ of $\I_A(\G)$. The set function $\G\mapsto \I_A(\G)$ is upper semicontinuous, thus there exists a neighborhood $\W'$ of $\G$ such that $\I_A(\G')\subseteq U$ for all $\G'\in\W'$. Moreover, by continuity of $\Phi_A$, we can suppose that $\Phi_A(\W')\subseteq\U$. The function
\[
\mathbb P\times U^\bot\ni(\G,\nu)\mapsto \G+\G_{\nu,0}
\]
is continuous, hence there exists a neighborhood $\W$ of $\G$ and a neighborhood $W$ of $0$ in $U^\bot$ such that $\W+\G_{W,0}\subseteq\W'$. Projections are open maps, thus the projection of $W$ on the cohomology contains a ball $B_\eps[U^\bot]$ centered at $0$. With these choices of $\W$ and $\eps$, let $\G'\in\W$ and $c\in c(\G')+B_\eps [U^\bot]$. We can then take as $\G''$ the pseudograph $\Phi_A(\G'+\G_{\nu,0})$ where $\nu\in W$ satisfies $[\nu]=c-c(\G')$. Indeed, by Remark \ref{remark laxoleinik}(iv) we find $N$ such that
\[
\G'\ \vdash_{0,\F}\ \G'_{|U}=\bigl(\G'+\G_{\nu,0}\bigr)_{|U}\ \vdash_{N,\F}\ \Phi_A\bigl(\G'+\G_{\nu,0}\bigr)=\G''.
\qedhere
\]
\end{proof}

The Lemma \ref{basic step} easily extends to operators in $\Sigma^\infty_c$.

\begin{proposition}
\label{basic limit step}
Let $\Phi_A\in\Sigma^\infty_c$. Then, for every $\G\in\mathbb P_{c}$ and for every neighborhood $\U$ of $\Phi_A(\G)$ in $\mathbb P$ there exist $N\in\N$, a neighborhood $\mathbb W$ of $\G$ and an $\eps>0$ such that:
\begin{gather*}
\forall\, \G'\in\mathbb W\,,\ c\in c(\G')+B_\eps R_A(\G) \qquad\exists\, \G''\quad\text{such that}
\\
\G''\in\U,\qquad\G'\ \vdash_{N,\F}\ \G'',\qquad c(\G'')=c.
\end{gather*}
\end{proposition}

\begin{proof}
Let us fix $\G$ and $\U$, and let us consider $\I_A(\G)$ and one of its adapted neighborhoods $U$. By Proposition \ref{Sigma properties} (iii) there exists $A'\in\sigma$ such that $\Phi_{A'}(\G)\in\U$ and $\I_{A'}(\G)\subseteq U$. This implies
\[
R_{A'}(\G)=[\I_{A'}(\G)^\bot]\supseteq [U^\bot]=[\I_A(\G)^\bot]=R_A(\G).
\]
We apply the Lemma \ref{basic step} and we get the result.
\end{proof}

In the following lemma we prove two similar results which show how Proposition \ref{basic limit step} has a good behavior under composition. The second version is in principle stronger but we will see that the first version would eventually lead to the same results, at least for an equi-semiconcave family $\F$. Therefore \textit{a posteriori} the second version is not strictly needed here.

The main point in both results is that, if we compose several operators in $\Sigma^\infty_c$, the set of allowed directions which we get is greater than just the union of the allowed directions obtained by applying separately Proposition \ref{basic limit step} to each operator. In fact, we obtain the vector subspace generated by this union.

\begin{lemma}
\label{concatenate}
Let $\Phi_{A_1},\dots,\Phi_{A_n}\in\Sigma_c^\infty$. Then: for every $\G\in\mathbb P$ and for every neighborhood $\U$ of $\Phi_{A_n}\circ\dots\circ\Phi_{A_1}(\G)$ in $\mathbb P$ there exist $N\in\N$, a neighborhood $\mathbb W$ of $\G$ and an $\eps>0$ such that:
\begin{gather*}
\forall\, \G'\in\mathbb W,\ \forall\, c\in c(\G')+B_\eps \Bigl(R_{A_1}(\G)+R_{A_2}(\Phi_{A_1}(\G))+\dots+R_{A_n}\bigl(\Phi_{A_{n-1}}\circ\dots\circ\Phi_{A_1}(\G)\bigr)\Bigr)
\\
\exists\, \G'':\qquad\G''\in\U,\qquad\G'\ \vdash_{N,\F}\ \G'',\qquad c(\G'')=c.
\end{gather*}

\medbreak

\noindent\textbf{Stronger version.\ \ }Under the same assumptions,
\begin{align*}
\forall\, \G'\in\mathbb W,&\quad \forall\ c\,\in\ \, c(\G')+B_\eps R_{A_1,\dots,A_n}(\G)
\\
&\exists\, \G'':\quad\G''\in\U,\qquad\G'\ \vdash_{N,\F}\ \G'',\qquad c(\G'')=c.
\end{align*}
This version is stronger because, in general, $R_{A_1,\dots,A_n}(\G)
$ may be strictly larger than $R_{A_1}(\G)+R_{A_2}(\Phi_{A_1}(\G))+\dots+R_{A_n}\bigl(\Phi_{A_{n-1}}\circ\dots\circ\Phi_{A_1}(\G)\bigr)$.
\end{lemma}

\begin{proof}
We suppose for simplicity $n=2$. The result is obtained by applying two times the Proposition \ref{basic limit step} and by noticing that if $Z',Z''$ are linear subspaces of a normed space and $\eps',\eps''>0$, then $B_{\eps'}Z'+B_{\eps''}Z''$ contains $B_\eps(Z'+Z'')$ for some $\eps>0$.
\end{proof}

\begin{proof}[Proof of the stronger version.]
Let us suppose $n=2$ for simplicity. Let us consider $\G\in\p$ and a neighborhood $\U$ of $\Phi_{A_2}\Phi_{A_1}(\G)$. Let $U_1$ and $U_2$ be adapted neighborhoods in $M$ of $\I_{A_2\circ A_1}(\G)$ and $\I_{A_2}(\Phi_{A_1}(\G))$ respectively. By Proposition \ref{Sigma properties}, there exist $A'_1,A'_2\in\sigma$ and a neighborhood $\W'$ of $\G$ in $\p$ such that
\[
\Phi_{A'_2}\Phi_{A'_1}(\G')\in\U, \quad\I_{A'_2\circ A'_1}(\G')\subseteq U_1\quad \text{and}\quad \I_{A'_2}(\Phi_{A'_1}(\G'))\subseteq U_2\qquad\forall\,\G'\in\W'.
\]

Let us now consider $\eta_1\in U_1^\bot$ and $\eta_2\in U_2^\bot$. Given $\G'=\G_{\eta,u}\in\W'$, we have
\[
\Phi_{A'_2}\bigl(\Phi_{A'_1}(\G'+\G_{\eta_1,0})+\G_{\eta_2,0}\bigr)=\G_{\eta+\eta_1+\eta_2,v}
\]
with
\[
v:=T_{A'_{2,\eta+\eta_1+\eta_2}}w\qquad w:=T_{A'_{1,\eta+\eta_1}}u.
\]
Let $x\in M$ be a point such that $dv_x$ exists. By Proposition \ref{pointwise basic forcing}, if $z$ is a point which realizes the minimum in the formula for $v(x)$, then $dw_z$ exists and
\[
dw_z+\eta_z+\eta_{1,z}+\eta_{2,z}\quad \vdash_{N_2,\F}\quad dv_x+\eta_x +\eta_{1, x} +\eta_{2, x}
\]
for some $N_2\in\N$. In the same way, if $y$ realizes the minimum in the formula for $w(z)$, then
\[
du_y+\eta_y+\eta_{1,y}\quad\vdash_{N_1,\F}\quad dw_z+\eta_z+\eta_{1,z}
\]
for some $N_1\in\N$. 

A generalization of the upper-semicontinuity result in \ref{remark laxoleinik}\,$(ii)$ shows that if $[\eta_1+ \eta_2]\in B_\eps[U_1^\bot+U_2^\bot]$ with $\eps$ small enough, then $y\in U_1$ and $z\in U_2$. We thus have $\eta_{1,y}=0$ and $\eta_{2,z}=0$ and therefore
\[
du_y+\eta_y\quad\vdash_{N_1+N_2,\F}\quad dv_x+\eta_x+\eta_{1,x}+\eta_{2,x}
\]
which is to say
\[
\G'=\G_{\eta,u}\quad \vdash_{N_1+N_2,\F}\quad \G_{\eta+\eta_1+\eta_2,v}.
\]
The proof is now completed with $\G''=\G_{\eta+\eta_1+\eta_2,v}$, up to choosing $\G'$ in a smaller neighborhood $\W\subseteq\W'$ in such a way that $\eps$ can be fixed independently of $\G'$.
\end{proof}

Note that, trivially, Lemma \ref{basic step} is a particular case of Proposition \ref{basic limit step}, which in turn is a particular case of Lemma \ref{concatenate}, hence just the latter will be used in the sequel.

\subsection{Heuristic application to twist maps}
\label{heuristic twist maps}

Even without the main general theorem \ref{theorem} of the next subsection, it is possible at this stage, using just the Lemma \ref{concatenate}, to derive some results about the presence of diffusion in polysystems of exact-symplectic twist maps on the cylinder. The discussion in this subsection will just be an heuristic one, even if everything could be made rigorous. The corresponding rigorous results will be proven in greater generality in the next subsection (Proposition \ref{R(c)=0} and Corollary \ref{d=1}).

Let $\F=\{H_1,H_2\}$ be a family of two Tonelli Hamiltonians on $\T\times\R$ and call $L_1,L_2$ the corresponding Lagrangians. Let us fix $c\in H^1(M,\R)$. We now show that just two scenarios are possible: 
\begin{itemize}
\item[-] there exists a circle of cohomology $c$ which is invariant for both $H_1$ and $H_2$ (which obviously provides an obstruction to diffusion); the only cohomology class forced by $c$ is $c$ itself; $R(c)=\{0\}$;
\item[-] there does not exists a circle as above; in this case $c$ forces a whole neighborhood of cohomology classes (and thus there exists diffusion in the sense of Proposition \ref{diffusion}); $R(c)=H^1(\T,R)\cong\R$.
\end{itemize}

Indeed, suppose that there exists a circle of cohomology $c$ which is invariant for both $H_1$ and $H_2$. It is standard that it can be identified with a pseudograph $\G$ which is invariant for both $\phi_{H_1}$ and $\phi_{H_2}$. In particular, by the very definition of forcing relation in Section \ref{forcing definition}, $\G$ is the only pseudograph forced by $\G$, and thus $c$ is the only cohomology class forced by $c$. Moreover, we deduce by Lemma \ref{concatenate} and the definition of $R(\G)$ that $R(\G)=\{0\}$, thus $R(c)=\{0\}$.

Vice versa, let us suppose that there does not exist such a common invariant circle. Let us consider $\G\in\p_c$, and let us apply $\Phi_{h_2}\circ\Phi_{h_1}$ to it ($h_1$ and $h_2$ are the Peierls barrier of $H_1,H_2$). By the first version of Lemma \ref{concatenate}, we get
\begin{equation}
\label{heuristic forcing}
\G\ \vdash_\F\ c+ B_\eps\bigl( R_{{h_1}}(\G)+R_{{h_2}}(\Phi_{h_1}(\G))\bigr)\qquad\forall\ \G\in\p_c.
\end{equation}
Recall that the image of $\Phi_{h_1}$ is contained in $\p$ and consists precisely of the weak Kam solutions for $H_1$, while the image of $\breve\Phi_{h_2}$ is contained in $\breve\p$ and consists of the dual weak Kam solutions for $H_2$. By assumption there do not exist common invariant circles, hence, in view of Proposition \ref{wks&dual},
\[
\Phi_{h_1}(\G)\neq\breve\Phi_{h_2}\Phi_{h_2}\Phi_{h_1}(\G).
\]
This implies (due to $d=1$) that $R_{{h_2}}(\Phi_{h_1}(\G))=H^1(\T,\R)$ regardless of $\G\in\p_c$. Thus the formula \eqref{heuristic forcing} implies that every $\G\in\p_c$ forces a whole neighborhood of cohomology classes. In that formula, $\eps$ depends in principle on $\G$, but one can show that by compactness it is possible to choose it uniformly in $\G$. Therefore $c$ forces a whole neighborhood of cohomology classes, as claimed. Since $R(\G)\supseteq R_{{h_2}}(\Phi_{h_1}(\G))$, the discussion also proves that $R(\G)=H^1(\T,\R)$ for every $\G\in\p_c$, thus $R(c)=H^1(\T,\R)$.
\medbreak

Notice how in this one-dimensional case our construction is optimal, in the following sense: the obstructions to the mechanism (i.e.\ the ``homological size'' of the sets $\I_{A}(\G)$) are real obstructions to the diffusion (i.e.\ the common invariant circles). On the contrary, in $d>1$ the construction will likely give just sufficient conditions for the diffusion: the obstructions to this mechanism may be circumvented by a different diffusion mechanism (such as, for instance, the Arnold mechanism presented by Bernard in \cite{Ber08}).

\subsection{A general theorem and some applications}
\label{section theorem}

We can summarize the argument used in Subsection \ref{heuristic twist maps} for the case $d=1$ by saying that we have applied Lemma \ref{concatenate} to $\Phi_{h_{2}}\circ\Phi_{h_{1}}$, and the result turned out to be optimal (so that there was no need to consider any other $\Phi\in\Sic$). Moreover, switching the order and considering $\Phi_{h_{1}}\circ\Phi_{h_{2}}$ would have led to the same result. The generalization of this argument to an arbitrary dimension $d$ is not completely straightforward: the choice of the operator could in principle make a difference, and it is less clear if the allowed directions which one obtains are optimal or not. 

In order to overcome these difficulties, we will adopt a slightly more abstract approach. This will give stronger conclusions, at the cost of a certain difficulty to interpret the obstructions which we will found.

We start with a ``raw'' result which follows immediately from Lemma \ref{concatenate}.

\begin{proposition}
Let $c$ be fixed. For any $\G\in\p_c$ and any finite string $s=(A_1,\dots,A_n)$ of elements of $\sic$, there exist $\eps(\G,s)>0$ such that
\begin{equation}
\label{raw}
c\ \vdash_\F\ c + \bigcap_{\G\in\p_c}\ \   \bigcup_{\substack{s=(A_1,\dots,A_n) \\ n\in\N}} B_{\eps(\G,s)}\bigl(R_{A_1,\dots,A_n}(\G)\bigr).
\end{equation}
\end{proposition}

\begin{proof}
Recall that $c\vdash_\F c'$ if and only if $\G\vdash_\F c'$ for all $\G\in\p_c$. The result is then a consequence of Lemma \ref{concatenate}.
\end{proof}

The general theorem \ref{theorem} will consist in a refined (but at the same time simplified) version of this raw result. Roughly speaking, it will be possible to replace the intersection over $\G\in\p_c$ with an intersection over a smaller set, to replace the union with a sum of vector subspaces and to choose $\eps$ uniformly in $(\G,s)$. This will simplify the right-hand side, and will lead in the end to a unique subspace of $H^1(M,\R)$ encoding all the information. In fact, $R(c)$ will be such a subspace. Moreover, exploiting some semicontinuity, the result will be proved to hold for $c'$ close enough to $c$. It will also be possible to replace the forcing relation $\vdash_\F$ with the mutual forcing relation $\dashv\vdash_\F$, and to have a locally uniform control on the $N$ appearing in its definition.

\medbreak
In order to motivate what follows, let us observe that the map $\G\mapsto R(\G)$ is by definition non-increasing along the action of elements of $\Sic$. More precisely,
\begin{equation}
\label{lyapunov}
R(\Phi(\G))\subseteq R(\G)\qquad\forall\ \G\in\p_c,\Phi\in\Sic.
\end{equation}
This can be interpreted by saying that this map is a sort of multi-valued Lyapunov function for the dynamics in $(\p_c,\Sic)$. Since we are interested in the set $R(c)$, which is the intersection of all the sets $R(\G)$, it is natural to look at the minimal components of the dynamics, whose properties have been analysed in Proposition \ref{minimal}.

For a minimal component $\m$ of $(\p_c,\Sic)$ let us define
\[
R(\m)=\bigcap_{\G\in\m}R(\G).
\]

\begin{proposition}[Equivalent expressions for $R(c)$]
\label{equivalent expressions}
\mbox{}
\begin{itemize}
\item[(i)] We have:
\begin{align*}
R(c)=\bigcap_{\m\text{ minimal}} R(\m).
\end{align*}
\item[(ii)] We have the following equivalent expressions for $R(\m)$:
\begin{align*}
R(\m)&=R(\G)=\sum_{\substack{A_1,\dots,A_n\in\sic \\ n\in\N}} R_{A_1,\dots,A_n}(\G)\qquad\text{ for any fixed }\G\in\m
\\
&=R_{A_1,\dots,A_n}(\G) \qquad\qquad\text{for some $A_1,\dots,A_n$ depending on $\G\in\m$}
\\
&=\sum_{\substack{\G\in\m \\ A\in\sic}}R_A(\G).
\end{align*} 
\end{itemize}
In particular, $R(\m)$ is a vector subspace for every $\m$, and the same holds for $R(c)$.
\end{proposition}

\begin{proof}
Let us prove item $(i)$. By the definition of $R(c)$ and $R(\m)$, it is clear that $R(c)\subseteq R(\m)$ for every minimal component $\m$, hence $R(c)\subseteq \cap_\m R(\m)$. For the reverse inclusion, let us notice that, since $\F$ is equi-semiconcave, by Proposition \ref{minimal}(ii) for every $\G\in\p_c$ there exists $\Phi\in\Sic$ such that $\Phi(\G)$ belongs to a minimal component. By \eqref{lyapunov},
\[
R(\G)\supseteq R(\Phi(\G))\supseteq \bigcap_\m R(\m).
\]
By taking the intersection over all $\G\in\p_c$, one gets the desired inclusion.

Let us now prove item $(ii)$. Thanks to relation \eqref{lyapunov} and the transitivity of minimal components (Proposition \ref{minimal}(iii)), we gets that the function $\G\mapsto R(\G)$ is constant on every minimal component. This proves that $R(\m)=R(\G)$ for any $\G\in\m$. Moreover, for every two strings $(A_1,\dots,A_n)$ and $(A'_1,\dots,A'_{n'})$ it holds
\begin{equation}
\label{eq:concat}
R_{A_1,\dots,A_n}(\G)+R_{A'_1,\dots,A'_{n'}}(\G)\subseteq R_{A_1,\dots,A_n,\bar A,A'_1,\dots,A'_{n'}}(\G)\subseteq R(\G)
\end{equation}
where $\bar A$ is any cost in $\sic$ such that $\Phi_{\bar A}\circ\Phi_{A_n}\circ\dots\circ\Phi_{A_1}(\G)=\G$. The existence of such cost $\bar A$ is guaranteed once again by the transitivity of the minimal component. This proves that
\[
R(\G)\supseteq \sum_{\substack{A_1,\dots,A_n\in\sic \\ n\in\N}} R_{A_1,\dots,A_n}(\G)\qquad\forall\,\G\in\m
\]
and the opposite inclusion is easy from the definitions. Moreover, since the dimension of $H^1(M,\R)$ is finite, we can write $R(\G)$ as a finite sum
\[
R(\G)=R_{A_1^1,\dots,A_{n_1}^1}(\G)+ R_{A_2^2,\dots,A_{n_2}^2}(\G)+\dots+ R_{A_1^N,\dots,A_{n_N}^N}(\G)
\]
and arguing as in \eqref{eq:concat} we get
\[
R(\G)=R_{A_1,\dots,A_n}(\G) \qquad\text{for some}\quad A_1,\dots,A_n\in\sic, n\in\N.
\]
The equality $R(\m)=\sum_{\G\in\m}\sum_{A\in\sic}R_A(\G)$ follows by similar arguments.
\end{proof}

Let us mention that, starting from the last expression for $R(\m)$ above, one can show that considering just the weaker version of Lemma \ref{concatenate} would eventually lead to the same results.

\begin{remark}
The function $\F\mapsto R_\F(c)$ is increasing. This is natural in view of the interpretation of $R_\F(c)$ as a set of allowed directions for diffusion, and follows by an inspection of the definitions (in fact, the map $\F\mapsto \Sic(\F)$ is also increasing).  In particular, let us point out that, since $R_\F(c)$ is a vector subspace,
\[
R_\F(c)\supseteq \sum_{H\in\F} R_{\{H\}}(c).
\]
The inclusion may be strict though: we will see that this is the case for two twist maps with non-common non-contractible invariant circles of cohomology $c$.
\end{remark}

We can now restate and prove Theorem \ref{theorem intro} of the Introduction. It is a generalization of Theorem 0.11 in \cite{Ber08} to the polysystem case.

\begin{theorem}
\label{theorem}
Let $\F$ be a family of one-periodic Tonelli Hamiltonians defined on the cotangent space of a boundaryless compact manifold $M$. Assume that $\F$ is equi-semiconcave in the sense of Definition \ref{equi-semiconcave}. Let $c\in H^1(M,\R)$. Then there exist a neighborhood $W$ of $c$ in $H^1(M,\R)$, $\eps>0$ and $N\in\N$ such that
\[
c'\ \dashv\vdash_{N,\F}\  c'+B_\eps R(c)\qquad\forall\,c'\in W.
\]
\end{theorem}

\begin{proof}
We subdivide the proof into four steps.

\medbreak

\emph{Step 1. For every $\m\subset\p_c$ minimal and every $\G\in\m$ there exist a neighborhood $\W_\G$ of $\G$ in $\p$, a natural number $N_\G$ and $\eps_\G>0$ such that
\[
\G'\ \ \vdash_{N_\G,\F}\ \  c(\G')+B_{\eps_\G} R(c) \qquad\qquad\forall\, \G'\in\mathbb W_\G.
\]
}
Let $\m$ be minimal and $\G\in\m$. Let $A_1,\dots,A_n\in\sic$ such that $R_{A_1,\dots,A_n}(\G)=R(\m)$. This is possible thanks to Proposition \ref{equivalent expressions}. Let us then apply Lemma \ref{concatenate} to $\G$ and to the composition $\Phi_{A_n}\circ\dots\circ\Phi_{A_1}$. Call $\W_\G$, $N_\G$ and $\eps_\G$ the objects yielded by that Proposition. We have
\begin{equation*}
\G'\ \ \vdash_{N_\G,\F}\ \  c(\G')+B_{\eps_\G} R(\m) \qquad\qquad\forall\, \G'\in\mathbb W_\G.
\end{equation*}
In particular, since $R(c)\subseteq R(\m)$, we have
\begin{equation*}
\G'\ \ \vdash_{N_\G,\F}\ \  c(\G')+B_{\eps_\G} R(c) \qquad\qquad\forall\, \G'\in\mathbb W_\G
\end{equation*}
as desired.

\medbreak
\emph{Step 2. For every $\G\in\p_c$ there exist a neighborhood $\W_\G$ of $\G$ in $\p$, a natural number $N_\G$ and $\eps_\G>0$ such that
\[
\G'\ \ \vdash_{N_\G,\F}\ \  c(\G')+B_{\eps_\G} R(c) \qquad\qquad\forall\, \G'\in\mathbb W_\G.
\]
}
Let $\G\in\p_c$. By Proposition \ref{minimal}, there exists $\Phi\in\Sic$ such that $\Phi(\G)$ is in a minimal component. Moreover, by Proposition \ref{Sigma properties} (iv) there exists $A\in\sigma$ such that $\Phi_A(\G)\in\W_{\Phi(\G)}$. By continuity, $\Phi_A(\G')\in\W_{\Phi(\G)}$ if $\G'$ is in a small enough neighborhood $\W_\G$ of $\G$. By Proposition \ref{sigma properties} and by Step 1, there exists $N_A\in\N$ such that
\[
\G'\ \ \vdash_{N_A,\F}\ \ \Phi_A(\G')\ \ \vdash_{N_{\Phi(\G)},\F}\ \  c(\G')+B_{\eps_{\Phi(\G)}} R(c)\qquad\qquad \forall\, \G'\in\mathbb W_\G.
\]
Thus we can take $N_\G=N_A+N_{\Phi(\G)}$ and $\eps_\G=\eps_{\Phi(\G)}$.

\medbreak
\emph{Step 3. There exist a neighborhood $W'$ of $c$ in $H^1(M,\R)$, a natural number $N$ and $\eps'>0$ such that 
\[
c'\ \ \vdash_{N,\F}\ \  c'+B_{\eps'} R(c)\qquad\forall\,c'\in W'.
\]
}
Let us choose $A_0$ in $\sigma$ (no matter which one, for instance $A_0=A^1_L$ with $L\in\F$ will work). The closure of $\Phi_{A_0}(\p_c)$ is compact, thus we can extract a finite subfamily $\{\G_j\}_j\subseteq\overline{\Phi_{A_0}(\p_c)}$ such that $\W=\cup_j \W_{\G_j}$ covers $\overline{\Phi_{A_0}(\p_c)}$. Moreover, it is true that $\W$ also covers $\Phi_{A_0}(\p_{W'})$ for a sufficiently small neighborhood $W'$ of $c$. Indeed, consider an arbitrary neighborhood $W''$ of $c$. The function $\G\mapsto c(\G)$ is continuous on the compact set $\overline{\Phi_{A_0}(\p_{W''})}\setminus \W$, hence its image is compact too. Since $c$ does not belong to this image, we can take as $W'$ the intersection of $W''$ with the complementary of the image.

In other words, for any $\G'\in\p_{W'}$ there exists $\bar j$ such that $\Phi_{A_0}(\G')\in\W_{\G_{\bar j}}$. Hence we obtain
\[
\G'\ \ \vdash_{N_{A_0},\F}\ \ \Phi_{A_0}(\G')\ \ \vdash_{\max_j N_{\G_j},\F}\ \ \ c(\G')+B_{\min_j \eps_{\G_j}} R(c)\qquad\forall\,\G'\in \p_{W'}.
\]
Thus we can take $N=N_{A_0}+\max_j N_{\G_j}$ and $\eps'=\min_j \eps_{\G_j}$, and the Step 3 is proved.

\medbreak
\emph{Step 4. There exist a neighborhood $W$ of $c$ in $H^1(M,\R)$, a natural number $N$ and $\eps>0$ such that
\[
c'\ \dashv\vdash_{N,\F}\  c'+B_\eps R(c)\qquad\forall\,c'\in W.
\]
}
In order to obtain the mutual forcing relation starting from the one-side forcing relation of Step 3, it suffices to take $W\subseteq W'$ and $\eps\le\eps' $ small enough in such a way that $W+B_\eps R(c)\subset W'$. This makes possible to apply the one-side forcing in the opposite direction. This concludes the proof of Step 4 (we keep the same $N$ as in the Step 3) and of the Theorem.
\end{proof}

\begin{remark}
A careful analysis of the proof of the theorem shows that the multi-valued function $c\mapsto R(c)$ is lower-semicontinuous: for any $c$ there exists a neighborhood $Z$ such that $R(c)\subseteq R(c')$ for every $c'\in Z$. Nevertheless, the statement of the theorem is somehow stronger, because it yields semicontinuity also on $N$ and $\eps$.
\end{remark}

In the remainder of this section we draw some relations between the subspace $R(c)$ and the underlying Hamiltonian polysystem dynamics.

\begin{proposition}
\label{R(c)=0}
Assume $\F$ equi-semiconcave. If there exists a $C^{1,1}$ $c$-weak Kam solution which is common to all $H\in\F$, then $R(c)=\{0\}$. If $d=1$ the viceversa holds: if $R(c)=\{0\}$ then all the Hamiltonians in $\F$ have an invariant circle in common. 
\end{proposition}

\begin{proof}
If there exists such a weak Kam solution as in the statement, we can identify it with a pseudograph $\G\in\p_c\cap\breve\p_c$ such that $\Phi_{A_L}(\G)=\G$ for every Lagrangian $L\in\F$. Since $\Sic$ is generated by such operators, we get that every $\Phi\in\Sigma^\infty_c$ satisfies $\Phi(\G)=\breve\Phi(\G)=\G$. The singleton $\{\G\}$ is thus a minimal set for $\p_c$ and, in view of formula \eqref{obstruction}, it satisfies $R(\{\G\})=\{0\}$.

On the other hand, if $d=1$ and $R(c)=\{0\}$, then there exists a minimal set $\m$ such that
\[
\sum_{\G\in\m,A\in\sigma_c^\infty} [\I_A(\G)^\bot]=\{0\},
\]
which means, thanks to the fact $d=1$ and to \eqref{obstruction}, that $\G=\breve\Phi_A\Phi_A(\G)$ for every $\G\in\m$ and $A\in\sigma_c^\infty$. Let us apply this to the Peierls barrier $h_{L,c}$ of a Lagrangian $L\in\F$. We get
\[
\G=\breve\Phi_{h_{L,c}}\Phi_{h_{L,c}}(\G)\in\, \text{Im}(\breve\Phi_{h_{L,c}})\qquad\forall\, L\in\F
\]
hence $\G$ is a dual weak Kam solution for every $L\in\F$, which in addition belongs to $\p$. This implies the result, by Proposition \ref{wks&dual}.
\end{proof}

We now can restate and prove the Corollary \ref{d=1 intro} about families of exact twist maps. The condition of equi-semiconcavity on $\F$ is dropped.

\begin{corollary} 
\label{d=1}
Let $M=\T=\R / \Z$. Let $\F$ be an arbitrary family of one-periodic Tonelli Hamiltonians on $T^*M\cong\T\times\R$. Let us make the identification $H^1(\T,\R)\cong\R$. If, for some $A<B\in H^1(\T,\R)$, the family $\F$ does not admit an invariant common circle with cohomology in $[A,B]$, then:
\begin{itemize}
\item[(i)] there exists an $\F$-polyorbit $(x_n,p_n)_{n\in\Z}$ satisfying $p_0 =A$ and $p_N= B$ for some $N\in\N$;
\item[(ii)] for every $H,H'\in\F$ and every $c,c'\in[A,B]$ there exists an $\F$-polyorbit $\alpha$-asymptotic to the Aubry set $\Aa_{H}(c)$ and $\omega$-asymptotic to $\Aa_{H'}(c')$
\item[(iii)] for every sequence $(c_i,H_i,\eps_i)_{i\in\Z}\subset [A,B]\times\F\times\R^+$ there exists an $\F$-polyorbit which visits in turn the $\eps_i$-neighborhoods of the Mather sets $\Mm_{H_i}(c_i)$.
\end{itemize}
\end{corollary}

\begin{proof}
If $\F$ is finite, the conclusion is immediate: by Proposition \ref{R(c)=0}, $R(c)=\R$ for every $c\in[A,B]$, hence by Theorem \ref{theorem} $[A,B]$ is contained in the same equivalence class for $\dashv\vdash_\F$. Therefore Proposition \ref{diffusion} applies, and allows to prove the results: for instance, in order to prove item $(i)$ one applies Proposition \ref{diffusion} (ii) with $\eta\equiv A$ and $\eta'\equiv B$.

If $\F$ is arbitrary, we just reduce to the case of $\F$ finite thanks to the following fact: if the family $\F$ does not admit invariant common circles with cohomology in $[A,B]$, then we can extract a finite subfamily $\F'\subset\F$ with the same property. Indeed, suppose on the contrary that every finite subfamily $\F'$ admits an invariant common circle with cohomology in $[A,B]$, and let us arbitrarily pick $H_0$ in $\F$: then the set $C(\F')$ defined by
\begin{align*}
C(\F')&=\left\{ \G\in\p_{[A,B]}: \G \text{ is a $C^{1,1}$ weak Kam solution for all } H\in\F'\cup\{H_0\}\right\}
\\
& =\bigcap_{H\in\F'\cup\{H_0\}}\Bigl(\{\G:\Phi_{A_H}(\G)=\G\}\cap\{\G:\breve\Phi_{A_H}(\G)=\G\}\Bigr)\ \cap\ \p_{[A,B]}
\end{align*}
is non-empty for all finite $\F'\subseteq\F$. Here $A_H$ denotes the time-one action of the Lagrangian associated to $H$. The second line in the above expression tells us that $C(\F')$ is closed and contained in the set $\Phi_{A_{H_0}}(\p_{[A,B]})$, which is relatively compact by Remark \ref{remark laxoleinik}$(iii)$. Hence $C(\F')$ is compact and not empty for every finite subfamily $\F'$.  We deduce that the sets $C(\F')$ satisfy the finite intersection property, because
\[
C(\F'_1)\cap\dots\cap C(\F'_n)= C(\F'_1\cup\dots\cup\F'_n)\neq\emptyset.
\]
By compactness, the whole intersection is non-empty too:
\[
\bigcap_{\substack{\F'\subseteq\F  \\  \F' \mathrm{ finite}}} C(\F')\neq\emptyset.
\]
Its elements are the invariant circles common to all the Hamiltonians of the family $\F$. This contradicts the assumptions.
\end{proof}

We end by discussing the application of Theorem \ref{theorem} to some special cases.

\begin{itemize}

\item[-] \emph{Case $\F=\{L\}$.} This is the case extensively treated in \cite{Ber08}. In that paper, $R(c)$ was defined as
\begin{equation}
\label{R(c) bernard}
R(c)=\bigcap_{\G\text{ $c$-weak Kam solution}} [\I_{{h_c}}(\G)^\bot],
\end{equation}
where $h_c$ is the $c$-Peierls barrier of $L$. Let us check that this definition coincides with the one given here. From Section \ref{Sic} we know that the minimal components in $\p_c$ are exactly the $c$-weak Kam solutions for $L$, and that ${h_c}\circ\sic=\sic\circ{h_c}={h_c}$. Therefore,
\[
\I_A(\G)\supseteq\I_{{h_c}\circ A}(\G)=\I_{{h_c}}(\G),\qquad\forall\ \G\in\p_c,A\in\sic
\]
(the first inclusion follows from \ref{compatibility}). The equality of \eqref{R(c) bernard} with our definition of $R(c)$ is then easy to verify.

Note that in this case the obstruction to the diffusion via the Mather mechanism is the homological size of $\I_{{h_c}}(\G)$, for every $c$-weak Kam solution $\G$. This set is also called the projected Aubry set of $\G$ (see Section \ref{wkt}), and taking the union over the $c$-weak Kam solutions $\G$ one gets the projection on $M$ of the Mañé set $\Nn(c)\subset T^*M$. A relation between $R(c)$ and the homology (in $T^*M$) of $\Nn(c)$ is given in \cite[Lemma 8.2]{Ber08}.

\item[-] \emph{Case $d=1$.} In this case the mutual forcing relation $\dashv\vdash_\F$ is well understood thanks to Proposition \ref{R(c)=0} and Theorem \ref{theorem}: if $B\subseteq H^1(\T,\R)$ is the closed set of those cohomology classes $c$ for which there exists a common invariant circle of cohomology $c$, then the equivalence classes for $\dashv\vdash_\F$ are the elements of $B$ and the connected components of its complementary.

\item[-] \emph{Commuting hamiltonians.} By the discussion in Section \ref{Sic}, we know that for every $c$ there exists a cost ${h_c}\in\sic$ such that ${h_c}\circ\sic=\sic\circ{h_c}={h_c}$. We also know that $h_c$ is the common Peierls barrier of all the Hamiltonians in $\F$, and that the minimal components in $\p_c$ are exactly the $c$-weak Kam solutions for one (hence all) Hamiltonian in $\F$. Arguing as in the case of a single Hamiltonian, one gets
\[
R_\F(c)=\bigcap_{\G\text{ $c$-weak Kam solution}} [\I_{{h_c}}(\G)^\bot],
\]
hence $R_\F(c)=R_{\{H\}}(c)$ for every $H\in\F$. Thus the obstructions are the same than those of each single Hamiltonian in $\F$. Therefore, the polysystem does not present any new kind of instability phenomena with respect to each system regarded separately (at least using the Mather mechanism presented here).

\item[-] \emph{General case. }The general situation is more complicated. Nevertheless, some information can still be extracted. For instance, let us suppose that $V$ is a one-dimensional subspace of $H^1(M,\R)$ not contained in $R(c)$. Then, there must exists a minimal component $\m$ such that $V$ is not contained in $R(\m)$. In particular, by Proposition \ref{equivalent expressions} and by invariance of $\m$, we have that
\begin{equation}
\label{V obstruction}
V\nsubseteq R_{A_1,\dots,A_n}(\G)\qquad\forall\,\G\in\m,\ \forall\,A_1,\dots,A_n\in\sic.
\end{equation}
By making different choices of $\G$ and $A_1,\dots,A_n$, one in principle gets a plethora of conditions on the dynamics of the family $\F$. In the next proposition we prove two sample statements, obtained by considering the two operators studied in Propositions \ref{switched} and \ref{switched peierls}.

Let us remark that the condition \ref{V obstruction} above essentially boil down to a condition on the various sets $\I_A(\G)=\G\wedge\breve\Phi_A\Phi_A(\G)$. By property (iv) in Proposition \ref{switched} we see that, at least for some choices of $\G$ and $A$, we can interpret $\G$ as an unstable manifold of some switched flow and $\breve\Phi_A\Phi_A(\G)$ as a stable manifold of another switched flow. Thus, at least for these choices of $\G$ and $A$, the size of the obstruction $\I_A(\G)=\G\wedge\breve\Phi_A\Phi_A(\G)$ has an interpretation as the size of the intersection between some unstable and stable manifolds.
\end{itemize}

\begin{proposition}
\label{sample}
Let $c\in H^1(M,\R)$. Suppose that $V$ is a one-dimensional subspace of $H^1(M,\R)$ not contained in $R(c)$.
\begin{itemize}
\item[(i)] For every arbitrary finite string $H_1,\dots,H_k$ of Hamiltonians in $\F$, there exists a subset $S\subset T^*M$ such that: $S$ is a Lipschitz graph over its projection on $M$, it is contained in a pseudograph of cohomology $c$, it is invariant (both in past and in future) for the switched flow
\[
\phi=\phi_{H_k}\circ\dots\circ\phi_{H_1},
\]
and its projection $\pi(S)\subseteq M$ satisfies
\[
V\nsubseteq[\pi(S)^\bot].
\]
\item[(ii)] For every pair of Hamiltonians $H_0,H_1\in\F$ there exists a $c$-weak Kam solution $\G_0$ for $H_0$ and a dual $c$-weak Kam solution $\G_1$ for $H_1$ such that
\[
V\nsubseteq\bigl[(\G_0\wedge\G_1)^\bot\bigr].
\]
Moreover, calling $h_0,h_1$ the Peierls barriers of $H_0$ and $H_1$, we can also suppose that $\Phi_{h_0}\Phi_{h_1}(\G_0)=\G_0$ and $\breve\Phi_{h_1}\breve\Phi_{h_0}(\G_1)=\G_1$.
\end{itemize}
\end{proposition}

\begin{proof}
Let $\m$ be a minimal component in $\p_c$ such that \eqref{V obstruction} holds true.
\begin{itemize}
\item[(i)] Call $A_1,\dots,A_k$ the time-one actions of $H_1,\dots,H_k$, and consider the composition
\[
A=A_k\circ\dots\circ A_1.
\]
Let us apply \eqref{V obstruction} with $n=1$ to the cost $h_{A_c}$ and to a fixed point $\G$ of $\Phi_{h_{A_c}}$ belonging to $\m$ (let us recall that, by invariance, every minimal component contains such a fixed point). We get
\[
V\nsubseteq[\I_{h_{A_c}}(\G)^\bot].
\]
Set $S=\G_{|\I_{h_{A_c}}(\G)}$. By a natural generalization of Proposition \ref{switched}, $S$ is invariant for $\phi$, thus the conclusion of item $(i)$ is achieved.
\item[(ii)] Call $h_{0},h_{1}$ the Peierls barrier at cohomology $c$ of $H_0$ and $H_1$ respectively. In \eqref{V obstruction} take $n=1$, $A_1={h_0\circ h_1}$ and $\G$ a fixed point of $\Phi_{A_1}$ in $\m$. We get
\[
V\nsubseteq[\I_{{A_1}}(\G)^\bot]=[\G\wedge \breve\Phi_{A_1} (\G)^\bot].
\]
Note that $\G$ is a $c$-weak Kam solution for $H_0$, because it belongs to the image of $\Phi_{h_0}$. Note also that $\breve\Phi_{A_1} (\G)$ is a dual $c$-weak Kam solution for $H_1$, because it belongs to the image of $\breve\Phi_{h_1}$ due to $\breve\Phi_{A_1}=\breve\Phi_{h_0\circ h_1}=\breve\Phi_{h_1}\circ\breve\Phi_{h_0}$. Hence the first part of the statement follows by setting $\G_0=\G$ and $\G_1=\breve\Phi_{A_1} (\G)$.

The second part of the statement follows by replacing, in the above argument, the cost $h_0\circ h_1$ with its Peierls barrier $h$, i.e.\ 
\[
h=\liminf_{m\to\infty}\ (h_{0}\circ h_{1})^m+m\,\alpha_{h_{0}\circ h_{1}}
\]
where $\alpha_{h_{0}\circ h_{1}}$ is the unique constant such that the $\liminf$ is real valued. The conclusion follows similarly as above, by taking in \eqref{V obstruction} $n=1$, $\Phi_1=\Phi_{h}$ and $\G$ a fixed point of $\Phi_h$ in $\m$, and then setting $\G_0=\G$, $\G_1=\breve\Phi_1 (\G)$. \qedhere
\end{itemize}
\end{proof}

\addcontentsline{toc}{section}{Acknowledgements}

\noindent\textbf{Acknowledgements. } 
I would like to thank Patrick Bernard for introducing me to this area as well as for helpful comments and discussions.

\end{document}